\documentclass[12pt,reqno]{amsart}
\usepackage[margin=1in]{geometry}       
\usepackage{graphicx}
\usepackage{amssymb}
\usepackage{aliascnt}
\usepackage{doi}
\usepackage{epstopdf}
\usepackage[dvipsnames]{xcolor}
\usepackage{tikz,pgfplots}
\usetikzlibrary{patterns.meta}
\usepgflibrary{shadings}
\usepgflibrary {shadings}
\usetikzlibrary {arrows.meta}

\newcommand\myarrowR[2]{
\draw[arrows = {-Stealth[length=7pt, inset=4pt, width=7pt]},  line width=0.7pt] #1 to[bend right] #2;}
\newcommand\myarrowL[2]{
\draw[arrows = {-Stealth[length=7pt, inset=4pt, width=7pt]},  line width=0.7pt] #1 to[bend left] #2;}

\definecolor{ballcolor}{HTML}{85B8E8} 
\definecolor{simplexcolor}{HTML}{DE6C6C}  
\definecolor{simplexpicturecolor}{HTML}{C73E3E} 
\definecolor{twoptcolor}{HTML}{6ECA4F} 
\definecolor{twoptcolorshading}{HTML}{ADE48B} 
\definecolor{symmbr}{HTML}{FFF594} 
\definecolor{mygray}{HTML}{EEEEEE} 
\definecolor{mymidgray}{HTML}{9E9E9E} 
\definecolor{mypurple}{HTML}{5D3A9B} 
\definecolor{myorange}{HTML}{E66100} 

\newcommand{\arxiv}[1]{%
\href{https://arxiv.org/abs/#1}{ArXiv:#1}}

\newtheorem{theorem}{Theorem}

\newaliascnt{lemma}{theorem}
\newtheorem{lemma}[lemma]{Lemma}
\aliascntresetthe{lemma}

\newaliascnt{proposition}{theorem}
\newtheorem{proposition}[proposition]{Proposition}
\aliascntresetthe{proposition}

\newaliascnt{corollary}{theorem}

\aliascntresetthe{corollary}

\newaliascnt{conjecture}{theorem}
\newtheorem{conjecture}[conjecture]{Conjecture}
\aliascntresetthe{conjecture}

\theoremstyle{definition}

\newcommand{\B}{{\mathbb B}}

\newcommand{\e}{\varepsilon}
\newcommand{\R}{{\mathbb R}}
\newcommand{\Rn}{{{\mathbb R}^n}}
\newcommand{\Rnp}{{{\mathbb R}^{n+1}}}

\DeclareMathOperator{\arcsinh}{arcsinh}
\DeclareMathOperator{\dist}{dist}
\DeclareMathOperator{\diam}{diam}
\DeclareMathOperator{\Length}{Length}
\DeclareMathOperator{\Area}{Area}
\DeclareMathOperator{\Vol}{Vol}

\DeclareMathOperator{\capzero}{Cap_0}
\DeclareMathOperator{\capone}{Cap_1}
\DeclareMathOperator{\captwo}{Cap_2}
\DeclareMathOperator{\capnegtwo}{Cap_{-2}}

\DeclareMathOperator{\capp}{Cap_\mathit{p}}
\DeclareMathOperator{\capq}{Cap_\mathit{q}}
\DeclareMathOperator{\capntwo}{Cap_{\mathit{n}-2}}

\title[Maximizing Riesz capacity ratios]{Maximizing Riesz capacity ratios: conjectures and theorems}
\author{Carrie Clark and Richard S. Laugesen}
\email{carriec2@illinois.edu, Laugesen@illinois.edu}
\address{University of Illinois, Urbana, IL 61801, USA}


\begin{document}

\begin{abstract}
A shape optimization program is developed for the ratio of Riesz capacities $\capq(K)/\capp(K)$, where $K$ ranges over compact sets in $\Rn$. In different regions of the $pq$-parameter plane, maximality is conjectured for the ball, the vertices of a regular simplex, or the endpoints of an interval. These cases are separated by a symmetry-breaking transition region where the shape of maximizers remains unclear. 

On the boundary of $pq$-parameter space one encounters existing theorems and conjectures, including: Watanabe's theorem minimizing Riesz capacity for given volume, the classical isodiametric theorem that maximizes volume for given diameter, Szeg\H{o}'s isodiametric theorem  maximizing Newtonian capacity for given diameter, and the still-open isodiametric conjecture for Riesz capacity. The first quadrant of parameter space contains P\'{o}lya and Szeg\H{o}'s conjecture on maximizing Newtonian over logarithmic capacity for planar sets. The maximal shape for each of these scenarios is known or conjectured to be the ball. 

In the third quadrant, where both $p$ and $q$ are negative, the maximizers are quite different: when one of the parameters is $-\infty$ and the other is suitably negative, maximality is proved for the vertices of a regular simplex or endpoints of an interval. Much more is proved in dimensions $1$ and $2$, where for large regions of the third quadrant, maximizers are shown to consist of the vertices of intervals or equilateral triangles. 
\end{abstract}

\maketitle

\setcounter{section}{-1}

\section{\bf Introduction}
\label{sec:intro}

What shape experiences the greatest change in Riesz capacity when the exponent $p$ in the interaction energy $|x-y|^{-p}$ increases or decreases? Specifically, consider the shape optimization problem of maximizing the Riesz capacity ratio 
\[
\frac{\capq(K)}{\capp(K)} 
\]
for fixed Riesz exponents $p \neq q$, as $K$ varies over all compact sets in $\Rn$. The ratio is scale and translation invariant and so the size or location of $K$ does not matter, only its shape. 

This family of maximization problems connects naturally to geometric quantities, because Riesz $p$-capacity with $p \to -\infty$ gives diameter while as $p \to n$ it yields volume. Thus special cases of the family yield known results and open problems, including: Watanabe's theorem on minimizing Riesz capacity for given volume, P\'{o}lya and Szeg\H{o}'s conjecture on maximizing the ratio of Newtonian over logarithmic capacity for planar sets, the classical isodiametric theorem that maximizes volume for given diameter, the isodiametric theorem maximizing Newtonian capacity for given diameter (which follows from capacitary Brunn--Minkowski), and the still-open isodiametric conjecture for Riesz capacity. 

The optimal set in each of these special cases is known or conjectured to be the ball. But this perfect symmetry is known to break in some regions of the $pq$-parameter plane, in particular when $p=-\infty$ (diameter) for certain $q$ values by recent work of Burchard, Choksi and Hess-Childs \cite{BCH20}. What are the maximizing shapes for the capacity ratio in the symmetry-breaking transition region and beyond?

\subsection*{Overview}
Motivated by these questions, we develop a wide-ranging shape optimization program for $\capq(K)/\capp(K)$, summarized in the next sections by \autoref{fig:pqdiagram1D}, \autoref{fig:pqdiagram2D} and \autoref{fig:pqdiagram} for dimensions $1$, $2$ and $\geq 3$, respectively. 

The figures show in which regions of the $pq$-parameter plane we can conjecture or prove maximality of the capacity ratio for the ball, regular $(n+1)$-point set (vertices of a regular simplex), or endpoints of an interval. Those parameter regions are separated by a ``terra incognita'' transition region where symmetry breaks and the shape of  maximizers is unclear.  

Progress is strongest in low dimensions, where our theorems cover substantial interior regions in \autoref{fig:pqdiagram1D} for $1$ dimension and \autoref{fig:pqdiagram2D} for the $2$-dimensional case. In these figures, a \textbf{two-point set} means the endpoints of a segment and a \textbf{regular three-point set} means the vertices of an equilateral triangle. Results in \autoref{sec:conjectures1} and \autoref{sec:conjectures2} are:
\begin{itemize}
\item[\small $\bullet$] ($1$-dim) interval maximizes when $q=1$: Watanabe's theorem for $-\infty<p<1$ (\autoref{th:onedim}),
\item[\small $\bullet$] ($1$-dim) interval is maximizer when $p \leq -1 < q < 1$ (\autoref{pr:onedim_upperleft}),
\item[\small $\bullet$] ($1$-dim) two-point set is maximal when $-1<p<0, q \leq -1$ (\autoref{pr:onedim_lowermiddle}),
\item[\small $\bullet$] ($2$-dim) disk is not maximal in a certain region where $p<-2, -2<q<-0.86$ (\autoref{th:symmetrybreaking2dim}),
\item[\small $\bullet$] ($2$-dim) regular three-point set is maximal when $p<q \leq -2$ (\autoref{th:2deqtriangle}), 
\item[\small $\bullet$] ($2$-dim) two-point set is maximal when $q<p<0$ and $q \leq -2$ (\autoref{th:2deqtriangle}). 
\end{itemize}
Minimization results follow too, by simply interchanging $p$ and $q$.  

Restricting temporarily to the class of planar three-point sets (vertices of a triangle), \autoref{th:trianglecapformula} evaluates the $p$-capacity explicitly for negative $p$ and \autoref{th:threeptratio} solves the ratio maximization problem completely for $p$ and $q$ both negative. 

Existence of a maximizing set when $p$ is negative is proved in \autoref{th:maximizer}, in all dimensions. We have not found an existence result when $p \geq 0$. 

Regarding the nature of the maximizing set in higher dimensions, \autoref{sec:conjectures} shows:
\begin{itemize}
\item[\small $\bullet$] ($n$-dim) ball is maximal when $-\infty\leq p \leq -2, q \in \{ n-2,n-1,n \}$ (\autoref{th:isodiametric}),
\item[\small $\bullet$] ($n$-dim) regular $(n+1)$-point set is maximal when $p=-\infty, q \leq -2$ (\autoref{pr:isodiamsimplex}),
\item[\small $\bullet$] ($n$-dim) two-point set is maximizer when $p<0, q=-\infty$ (\autoref{le:twopoint}). 
\end{itemize}
Here Riesz capacity with parameter $-\infty$ means the diameter of the set. The first of these three results is Szeg\H{o}'s isodiametric theorem for Newtonian capacity ($q=n-2$), and includes the classical isodiametric theorem for volume ($q=n$). The second and third results provide isodiametral maximization and minimization for Riesz capacity with parameters $q \leq -2$ and $p<0$, respectively.  

Much remains unknown in higher dimensions. Perhaps the most difficult question in the paper is \autoref{conj:riesz}, which asserts that the ball maximizes the capacity ratio when $0 \leq p<q$. P\'{o}lya and Szeg\H{o}'s \autoref{conj:ps45} is the case $(p,q,n)=(0,1,2)$, and Watanabe's theorem is the limiting case with $n-2<p<n$ and $q \nearrow n$. The conjectures are confirmed in \autoref{sec:examples} for some families of explicit examples. 

For negative parameters $p < q \leq -2$, we raise in \autoref{conj:simplex} that the maximizer should be a regular $(n+1)$-point set. So far, we can prove only the planar case, in \autoref{th:2deqtriangle} mentioned above. \autoref{conj:twopoint} claims for $q<p<0$ that the maximizer is a two-point set. 

\subsection*{Capacity, diameter, and volume} We now summarize the definitions of energy and capacity and the limiting cases that yield diameter and volume. References for these definitions and facts are collected in \cite[Appendix A]{CL24b}. 

Consider a compact nonempty set $K \subset \R^n, n \geq 1$. When $p \neq 0$, the Riesz $p$-energy of $K$ is
\[
V_p(K) = 
\begin{cases}
\min_\mu \int_K \! \int_K |x-y|^{-p} \, d\mu(x) d\mu(y) , & p > 0 , \\
\max_\mu \int_K \! \int_K |x-y|^{-p} \, d\mu(x) d\mu(y) , & p < 0 , 
\end{cases}
\]
where the minimum and maximum are taken over all probability measures on $K$, that is, positive unit Borel measures. The minimum or maximum is attained by an ``equilibrium" measure. When $p>0$, the energy is positive or $+\infty$, and if the energy is finite then the equilibrium measure is unique. When $p<0$, the energy is always finite, and it is positive provided $K$ contains more than one point. If $-2<p<0$ then the equilibrium measure is unique. In cases where the equilibrium measure is unique, we call it the $p$-equilibrium measure of $K$. For the empty set we define $V_p(\emptyset)$ to equal $+\infty$ when $p>0$ and $0$ when $p<0$. 

Riesz energy is not defined when $p=0$. Instead, we consider the logarithmic energy 
\[
V_{log}(K) = \min_\mu \int_K \! \int_K \log \frac{1}{|x-y|} \, d\mu(x) d\mu(y) ,
\]
where the minimum is taken over all probability measures on $K$. The energy is greater than $-\infty$ since $|x-y|$ is bounded. The minimum is attained by an equilibrium measure. If the energy is less than $+\infty$ then this logarithmic equilibrium measure is unique and is called the $0$-equilibrium measure. For the empty set, define $V_{log}(\emptyset)=+\infty$. 

The Riesz $p$-capacity of $K$ is 
\begin{equation*} 
\capp(K) = 
\begin{cases}
V_p(K)^{-1/p} , & p \neq 0 , \\
\exp(-V_{log}(K)) , & p = 0 .
\end{cases}
\end{equation*}
The $0$-capacity is also called logarithmic capacity. The capacity of the unit ball is plotted as a function of $p$ in \autoref{fig:ballplots}, for the first few dimensions $n$. Ellipsoidal examples are in \autoref{sec:examples}, and for arbitrary three-point sets in the plane when $p<0$, see \autoref{th:trianglecapformula}. 
\setcounter{figure}{-1} 
\begin{figure}
\begin{center}
\includegraphics[width=0.5\textwidth]{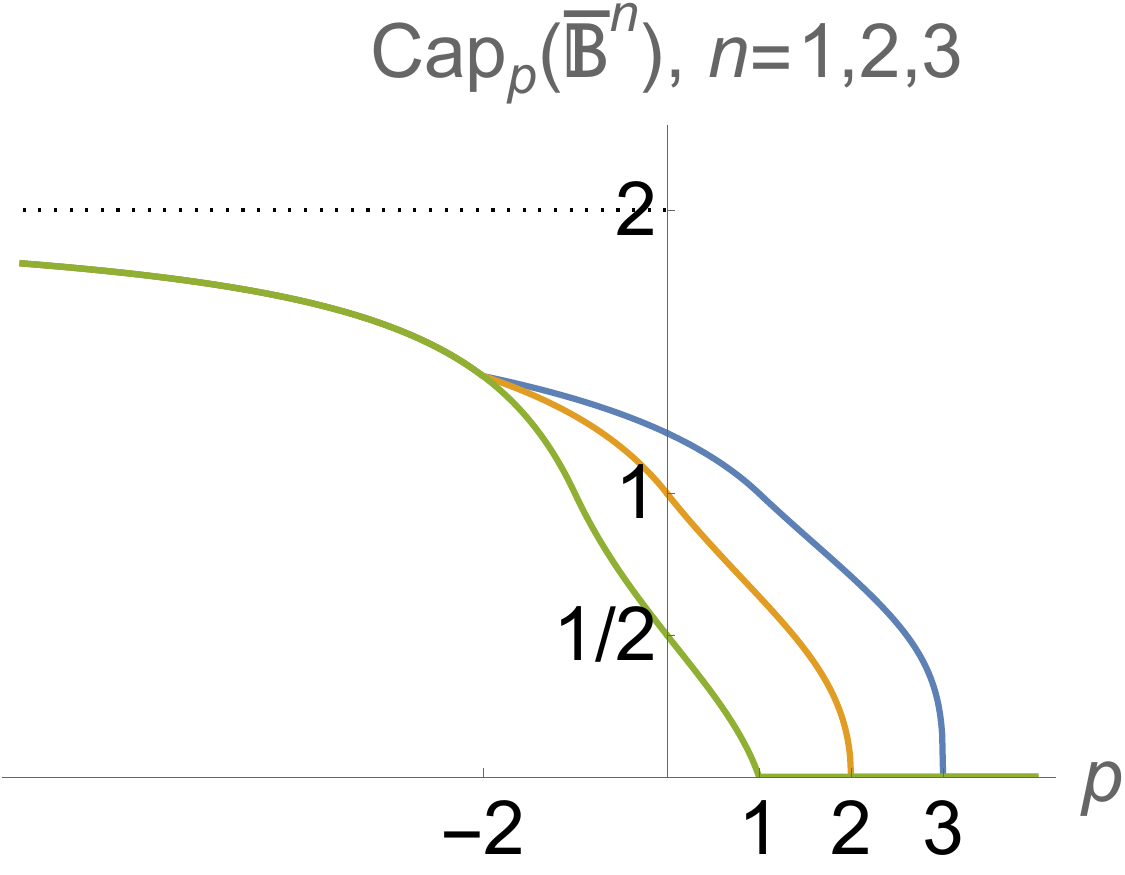}
\end{center}
\caption{\label{fig:ballplots} The Riesz capacity $\capp(\overline{\B}^n)$ of the closed unit ball, plotted as a function of $p$ for the first three values of $n$. The vertical intercept at $p=0$ is the logarithmic capacity $\capzero(\overline{\B}^n)$, with values $1/2, 1, 2e^{-1/2}$ for $n=1,2,3$ respectively. When $p \leq -2$, the capacity is $2^{1+1/p}$ for each $n$, with limiting value $2$ (the diameter of the ball) as $p \to -\infty$. The plots are based on standard capacity formulas \cite[Appendix A]{CL24b}. 
}
\end{figure}

When $p \geq 0$, the Riesz capacity is positive if and only if the energy is finite. When $p<0$, the capacity is positive provided $K$ contains more than one point. Newtonian energy $V_{n-2}(K)$ and Newtonian capacity $\capntwo(K)$ are the special case where 
\[
n \geq 3 \quad \text{and} \quad p=n-2. 
\]
The classical electrostatic capacities are Newtonian when $n=3$ and logarithmic for $n=2$. 

\subsection*{Known properties} Capacity is monotonic with respect to set inclusion, since $K_1 \subset K_2$ implies $\capp(K_1) \leq \capp(K_2)$, and capacity scales linearly with 
\[
\capp(sK)=s \capp(K) , \qquad s>0 .
\]
Diameter and volume are endpoint cases of Riesz capacity by the next result, in which volume (Lebesgue measure) on $\Rn$ is denoted $\Vol_n$. 
\begin{theorem}[Clark and Laugesen \cite{CL24b}] \label{th:Rieszmonotonicity} 
Suppose $K$ is a compact subset of $\Rn, n \geq 1$. 

(a) (Monotonicity) $\capp(K)$ is a decreasing function of $p \in \R$, and is strictly decreasing on the interval where it is positive. 

(b) (Diameter: $p = -\infty$)  
\[
\diam(K) = \lim_{p \to -\infty} \capp(K) ,
\]
and $2^{1/p} \leq \capp(K)/\diam(K) \leq 1$ for all $p <0$. If $K$ is $1$-dimensional, $K \subset \R$, then $2^{1/p} \diam(K) = \capp(K)$ for all $p \leq -1$.

(c) (Volume: $p = n$) 
\[
\Vol_n(K)^{1/n} = |{\mathbb S}^{n-1}|^{1/n} \lim_{p \nearrow n} \frac{\capp(K)}{(n-p)^{1/p}} .
\]

(d) For $p \geq n$, one has $\capp(K)=0$. 
\end{theorem}
Part (d) is well known by Hausdorff dimension estimates \cite[Theorem 4.3.1]{BHS19}. 

\subsection*{$k$-point sets} By a $k$-point set we mean a finite set in Euclidean space containing exactly $k$ elements. The set is \textbf{regular} if its points form the vertices of a regular simplex. The $p$-equilibrium measure on a regular $k$-point set is equidistributed with measure $1/k$ at each point, when $p<0$, and the set has $p$-capacity $\left( (k-1)/k \right)^{-1/p} d$ where $d$ is the diameter. See the regular and nonregular examples in \autoref{sec:threepoint}. 

\subsection*{Related literature} 

 The ratio of Riesz capacities studied in this paper resembles an attraction--repulsion energy, in some respects. Recent work by Carrillo, Figalli, and Patacchini \cite{CFP17} and Frank and Lieb \cite{FL18} can serve as an introduction to that area. There the goal is to minimize an energy whose kernel encourages short-range repulsion and long-range attraction of particles. A ``difference of power laws'' kernel $|x-y|^\alpha/\alpha-|x-y|^\beta/\beta$ where $\alpha>\beta>-n$ is a key model example. 
 
The rich theory of that model has been explored in particular by Lim and McCann \cite{LM21}, later joined by Davies \cite{DLM22,DLM23}. Those papers establish in certain regions of the $\alpha\beta$-parameter plane that the minimizer is the uniform measure on either a sphere or the vertices of a regular simplex --- the same extremal shapes expected in the current paper. The link is direct when $\alpha \to \infty$ (the ``hard confinement limit''), for then the long-range potential imposes a diameter constraint that corresponds in this paper to the limit $p \to -\infty$. 
 
A crucial difference, though, is that extremizing the attractive--repulsive energy involves optimizing with respect only to a single measure, whereas in this paper, extremizing the ratio of $q$-capacity to $p$-capacity requires simultaneous treatment of different measures in the numerator and denominator. 

Other papers that have indirectly inspired the current research include Jerison's solution of the Minkowski problem for Newtonian capacity \cite{J96}, which relies on Borell's capacitary Brunn--Minkowski inequality \cite{B83}, Xiao's result \cite{X17}  on the P\'{o}lya--Szeg\H{o} double-disk conjecture for minimizing Newtonian capacity among convex sets of given surface area, Laugesen's extremal result for Riesz capacity under linear transformations \cite{L22}, and Solynin and Zalgaller's beautiful theorem that among $N$-gons of given area, the regular $N$-gon minimizes logarithmic capacity \cite{SZ04}.

\section{\bf Riesz vs.\ Riesz: capacity ratios in $1$ dimension}
\label{sec:conjectures1}

First, we show Riesz capacity is minimal for the interval among sets of given length (Lebesgue measure), which can be regarded as a $1$-dimensional version of Watanabe's theorem. It is an edge case of \autoref{conj:riesz} later, as indicated in \autoref{fig:pqdiagram1D}. The high-numbered conjectures and results in that figure apply in all dimensions, and are stated later in \autoref{sec:conjectures}. 
\begin{theorem}[Maximizing length divided by $p$-capacity, i.e.\ minimizing $p$-capacity for given length] \label{th:onedim}
Let $p<1$. The ratio of length to $p$-capacity is maximal for a closed interval among compact sets $K \subset \R$, meaning  
\begin{equation} \label{eq:caplength}
\capp(K) \geq \frac{\capp(\overline{\B}^1)}{2} \Length(K) .
\end{equation}
When $K$ has positive length, equality holds in \eqref{eq:caplength} if and only if $K=J \cup Z$ for some nondegenerate closed interval $J$ where  
\[
\begin{cases}
\text{$Z$ has inner $p$-capacity zero,} & 0 \leq p < 1 , \\
\text{$Z$ is empty,} & p < 0 .
\end{cases}
\]
When $K$ has zero length, equality holds in \eqref{eq:caplength} if and only if $K$ has $p$-capacity zero, which when $p < 0$ means that $K$ is either empty or a singleton.
\end{theorem}
See \autoref{sec:onedimproof} for the proof.


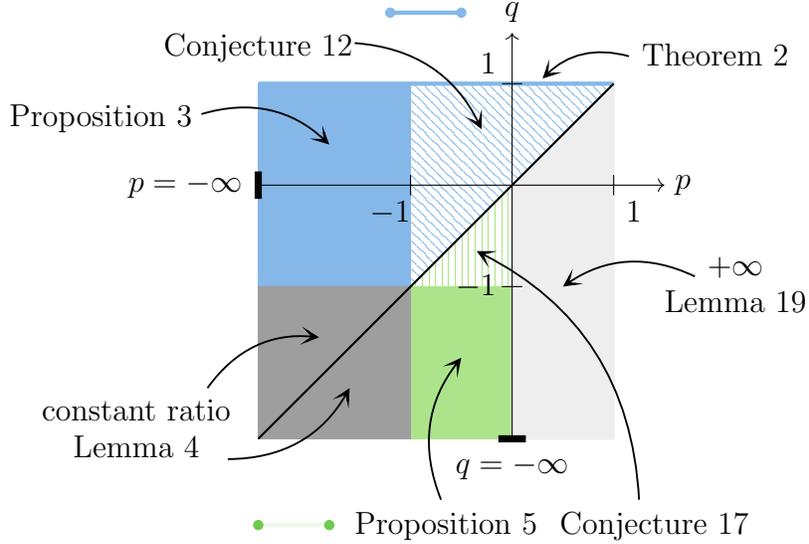
\begin{figure}
\begin{tikzpicture}[scale=1.35]

\def\n{1}
\def\pmax{1.5}
\def\pmin{-2.5}
\def\qmax{1.5}
\def\qmin{-2.5}
        

\filldraw[pattern={Lines[angle=-45,distance=2.5pt]}, pattern color=ballcolor, draw=none] (0,0) -- (\n, \n) -- (-1, \n) -- (-1, -1) -- cycle;
\filldraw[ballcolor , draw=ballcolor] (\pmin, \n) -- (-1, \n) -- (-1, -1) -- (\pmin,-1) -- cycle;
\filldraw[pattern={Lines[angle=90,distance=2.5pt]}, pattern color=twoptcolorshading, draw=none] (0,0) -- (-1, -1) -- (0, -1);
\filldraw[color=twoptcolorshading, draw=twoptcolorshading] (-1,\qmin) -- (0,\qmin) -- (0,-1) -- (-1,-1) -- cycle;
\filldraw[color=mymidgray, draw=mymidgray] (\pmin,\qmin) -- (-1,\qmin) -- (-1,-1) -- (\pmin,-1) -- cycle;
\filldraw[color=mygray, draw=mygray,  thick] (\n,\n) -- (\n,\qmin) -- (0,\qmin) -- (0,0) -- cycle;

\draw[ballcolor, ultra thick]  (\pmin,\n) -- (\n,\n);


\draw (\pmin,1.6)   node[below] {\autoref{conj:riesz}};
\myarrowL{(-\n-0.55,1.4)}{(-.3,\n-0.45)}
\draw (\pmin-1.55,0.9)   node[below] {\autoref{pr:onedim_upperleft}};
\myarrowL{(\pmin-0.56,0.7)}{(\pmin+0.7,0.4)}
\draw (\n+1,1.5)   node[below] {\autoref{th:onedim}};
\myarrowR{(\n+0.15,1.27)}{(\n-0.7,\n+0.05)}
\def\tpx{-1.2} 
\def\tpy{\n+0.7} 
\def\tpl{0.7} 
\draw[ballcolor, ultra thick] (\tpx,\tpy) -- (\tpx+\tpl,\tpy);
\filldraw[ballcolor] (\tpx,\tpy) circle (0.25ex);
\filldraw[ballcolor] (\tpx+\tpl,\tpy) circle (0.25ex);

\draw (1.4,\qmin-1.1)   node[above] { \autoref{conj:twopoint}};
\myarrowR{(1.25,\qmin-0.6)}{(-0.35,-0.65)}
\draw (-0.65,\qmin-1.1)   node[above] { \autoref{pr:onedim_lowermiddle}};
\myarrowL{(-0.7,\qmin-0.6)}{(-0.5,-1.7)}
\def\tpx{-2.5} 
\def\tpy{\qmin-0.85} 
\def\tpl{0.7} 
\draw[twoptcolor!12, ultra thick] (\tpx,\tpy) -- (\tpx+\tpl,\tpy);
\filldraw[twoptcolor] (\tpx,\tpy) circle (0.25ex);
\filldraw[twoptcolor] (\tpx+\tpl,\tpy) circle (0.25ex);

\draw (\pmin-1.2,-2) node[below,align=center] {constant ratio\\\autoref{le:onedim_lowerleft}};
\myarrowL{(\pmin-0.5,-2)}{(\pmin+0.6,-1.5)}
\myarrowR{(\pmin-0.3,-2.7)}{(\pmin+0.9,-2)}

\draw (\n+1.2,-0.6)  node[below,align=center] {$+\infty$\\ \autoref{le:maxinfty}};
\myarrowR{(\n+0.8,-0.9)}{(\n-0.5,-1)}


\draw[black,  thick]  (\qmin,\pmin) to (\n,\n);
\draw[->] (\pmin,0) to (\pmax,0);
\draw (\pmax,0)  node[right] {$p$};
\draw[->] (0,\qmin) to (0,\qmax);
\draw (0,\qmax)  node[above] {$q$};
\draw[-] (\n,-0.1) to (\n,0.1);
\draw (\n+0.2, -0.05) node[below] {$\n$};
\draw[-] (-1,-0.1) to (-1,0.1);
\draw (-1.2, -0.05) node[below] {$-1$};
\draw[-] (-0.1,\n) to (0.1,\n);
\draw (-0.05,\n+0.2) node[left] {$\n$};
\draw[-] (-0.1,-1) to (0.1,-1);
\draw (-0.05,-1) node[left] {$-1$};
\draw (\pmin-0.05,0)  node[left] {$p=-\infty$};
\draw[black,fill=black] (\pmin-0.03,-0.13)  rectangle  (\pmin+0.03,0.13) ;
\draw (0,\qmin-0.05)  node[below] {$q=-\infty$};
\draw[black,fill=black] (-0.13,\qmin-0.03)  rectangle  (0.13,\qmin+0.03) ;
\end{tikzpicture}
\caption{\label{fig:pqdiagram1D} ($n=1$) Maximizing $\capq(K)/\capp(K)$ for $K \subset \R$. Solid regions and segments on the boundary indicate rigorous results. Striped regions are conjectural. Blue corresponds to the interval and green to the two-point set.}
\end{figure}


Next, by a short argument in \hyperlink{proofof_pr:onedim_upperleft}{\autoref*{sec:additionalproofs}} we resolve the isodiametric problem for Riesz capacity in $1$ dimension. 
\begin{proposition}[Maximizing $q$-capacity given diameter] \label{pr:onedim_upperleft}
Let $p \leq -1 < q < 1$. Among compact sets $K \subset \R$ that contain more than one point, the ratio $\capq(K)/\capp(K)$ is maximal if and only if $K$ is an interval. Equivalently, 
\[
\capq(K) \leq \frac{\capq(\overline{\B}^1)}{2} \diam(K) , \qquad q \in (-1,1) ,
\]
with equality if and only if $K$ is an interval. 
\end{proposition}

The bottom left region of \autoref{fig:pqdiagram1D} is uninteresting in $1$ dimension, since by the next lemma, the ratio of capacities is the same for all sets. 
\begin{lemma} \label{le:onedim_lowerleft}
Let $p,q \leq -1$. If $K \subset \R$ is compact and contains more than one point then 
\[
\frac{\capq(K)}{\capp(K)} = 2^{1/q-1/p} .
\]
\end{lemma}
For the proof, simply recall from \autoref{th:Rieszmonotonicity}(b) that $\capp(K)=2^{1/p} \diam(K)$, when $p \leq -1$, and similarly for the $q$-capacity. 

Finally, in $1$ dimension we prove a substantial portion of the ``two-point'' \autoref{conj:twopoint}. 
\begin{proposition}[Minimizing $p$-capacity for given diameter] \label{pr:onedim_lowermiddle}
Let $-1<p<0, q \leq -1$. Among compact sets $K \subset \R$ that contain more than one point, 
\[
\frac{\capq(K)}{\capp(K)} \leq 2^{1/q-1/p} 
\]
with equality if and only if $K$ contains exactly two points. Equivalently, 
\[
\capp(K) \geq 2^{1/p} \diam(K) , \qquad p \in (-1,0) , 
\]
with equality if and only if $K=\{a,b\}$ for some $a<b$. 
\end{proposition}
\hyperlink{proofof:pr:onedim_lowermiddle}{\autoref*{sec:additionalproofs}} shows the proof. The inequality itself is easy, while the ``only if'' equality statement requires a little care.

\section{\bf Riesz vs.\ Riesz for $2$ dimensions}
\label{sec:conjectures2}

In $2$ dimensions we can handle significant parts of the parameter plane. \autoref{fig:pqdiagram2D} summarizes the findings. Again, high-numbered conjectures and results apply in all dimensions and are found in the next section. 


\begin{figure}
\begin{tikzpicture}[scale=1.25]

\def\n{2}
\def\pmax{3}
\def\pmin{-4}
\def\qmax{3}
\def\qmin{-4}


\filldraw[top color=ballcolor!40, bottom color=white, middle color=ballcolor!5, shading angle=10, draw=none] (\pmin,\n-2.4) -- (\pmin,\n-2) -- (0,\n-2) -- (-0.75,-0.75) -- (-\n,-0.8);
\filldraw[pattern={Lines[angle=-45,distance=3pt,line width=2.4pt, xshift=2pt]}, pattern color=white, draw=none] (\pmin+0.02,\n-3) -- (\pmin-0.2,\n-2) -- (0,\n-2) -- (-0.65,-0.65);
\filldraw[pattern={Lines[angle=-45,distance=3pt,line width=2.4pt, xshift=2pt]}, pattern color=white, draw=none] (\pmin+0.02,\n-3) -- (\pmin-.2,\n-2) -- (0,\n-2) -- (-0.65,-0.65);
\draw[white, ultra thick] (\pmin,\n-2) -- (\pmin,\n-2.4) -- (-\n,-0.8) -- (-0.75,-0.75);
\filldraw[pattern={Lines[angle=-45,distance=3pt, line width=0.5pt]}, pattern color=ballcolor, draw=none] (0,0) -- (\n, \n) -- (\pmin, \n) -- (\pmin, \n-2) -- (0, \n-2);
\filldraw[color=mygray, draw=mygray] (\n,\n) -- (\n,\qmin) -- (0,\qmin) -- (0,0);
\filldraw[pattern={Lines[angle=90,distance=3pt, line width=0.5pt]}, pattern color=twoptcolorshading, draw=none] (-2,-2) -- (0, -2) -- (0, 0)  ;
\filldraw[color=twoptcolorshading, draw=none] (\pmin,\qmin) -- (-2, -2) -- (0,-2) -- (0,\qmin)  ;
\filldraw[color=simplexcolor] (\pmin,\qmin) -- (\qmin, -2) -- (-2, -2) ;
\filldraw[top color=white, bottom color=symmbr, middle color=symmbr!60, shading angle=-29, draw=white] (\pmin,-0.9) -- (-2,-2) -- (-1.6,-1.6) -- (-2.5,-0.9) -- (-3.5,-0.5) -- (\pmin,-0.42);
\filldraw[pattern={Lines[angle=45,distance=3pt,line width=2.4pt, xshift=2pt]}, pattern color=white, draw=none]  (\pmin,-0.9) -- (-2,-2) -- (-1.6,-1.6) -- (-2.5,-0.9) -- (-3.5,-0.5) -- (\pmin,-0.42);
\filldraw[pattern={Lines[angle=45,distance=3pt,line width=2.4pt, xshift=2pt]}, pattern color=white, draw=none]  (\pmin,-0.9) -- (-2,-2) -- (-1.6,-1.6) -- (-2.5,-0.9) -- (-3.5,-0.5) -- (\pmin,-0.42);
\draw[white] (\pmin,-0.9) -- (\pmin,-0.42);
\filldraw[color=symmbr] plot[] coordinates {(-2., -2.) (-2.0184, -1.93555) (-2.03957, -1.87111) (-2.06393, -1.80666) (-2.09199, -1.74222) (-2.12439, -1.67778) (-2.16191, -1.61333) (-2.20551, -1.54889) (-2.25643, -1.48444) (-2.31628, -1.42) (-2.38713, -1.35555) (-2.47179, -1.29111) (-2.57407, -1.22666) (-2.69934, -1.16222) (-2.8554, -1.09778) (-3.05398, -1.03333) (-3.31366, -0.968887) (-3.66564, -0.904442) (-4, -0.865776) (-4,-2)};


\draw[black,  thick]  (\qmin,\pmin) to (\n,\n);

\draw[->] (\pmin,0) to (\pmax,0);
\draw (\pmax,0)  node[right] {$p$};

\draw[->] (0,\qmin) to (0,\qmax);
\draw (0,\qmax)  node[above] {$q$};

\draw[-] (\n,-0.1) to (\n,0.1);
\draw (\n+0.2, -0.1) node[below] {$\n$};
\draw[-] (-2,-0.1) to (-2,0.1);
\draw (-2, -0.1) node[below] {$-2$};

\draw[-] (-0.1,-2) to (0.1,-2);
\draw (-0.1,-2) node[left] {$-2$};
\draw[-] (-0.1,\n-1) to (0.1,\n-1);
\draw (-0.1,\n-1) node[left] {$1$};
\draw[-] (-0.1,\n) to (0.1,\n);
\draw (-0.1,\n+0.2) node[left] {$\n$};

\draw[ballcolor,fill=ballcolor] (\pmin,\n-2) circle (.32ex);

\draw (\pmin-0.05,0)  node[below] {$p=-\infty$};

\draw (0,\qmin-0.05)  node[below] {$q=-\infty$};
\draw[black,fill=black] (-0.13,\qmin-0.03)  rectangle  (0.13,\qmin+0.03) ;


\draw[-, ballcolor, ultra thick] (\n-2,\n) to (\n,\n);
\draw[-, ballcolor, ultra thick] (\pmin,\n) to (-2,\n);
\draw[-, ballcolor, ultra thick] (\pmin,\n-1) to (-2,\n-1);
\draw[-, ballcolor, ultra thick] (\pmin,\n-2) to (-2,\n-2);
\draw[ballcolor, dash pattern=on 5pt off 5.1pt, thick] (\pmin,\n-2.4)  .. controls (\pmin+0.65,\n-2.5) and (\pmin+0.95,\n-2.5) .. (\pmin+1.5,\n-2.9);
\draw  (\pmin+1.55,\n-3) node[right] {?};
\draw[ballcolor, very thick] (\n-2+0.08,\n-1+0.08) -- (\n-2-0.08,\n-1-0.08);
\draw[ballcolor, very thick] (\n-2-0.08,\n-1+0.08) -- (\n-2+0.08,\n-1-0.08);
\draw[ballcolor,fill=ballcolor] (\pmin,\n) circle (.32ex);
\draw[ballcolor,fill=ballcolor] (\pmin,\n-2) circle (.32ex);
\draw[ballcolor,fill=ballcolor] (\pmin,\n-1) circle (.32ex);


\draw (\pmin-1.8,-1.1) node[above] {symmetry breaking};
\draw (\pmin-1.8,-1.45) node[above] {\autoref{th:symmetrybreaking2dim}};
\myarrowL{(\pmin-0.8,-1.2)}{(\pmin+0.7,-1.6)}

\draw (\pmin-2,\qmin+0.2)   node[above] {\autoref{th:2deqtriangle}(a)};
\myarrowL{(\pmin-0.87,\qmin+0.5)}{(\pmin+0.6,\qmin+1.2)}
\draw (\pmin-2.1,\qmin+0.9)   node[above] {\autoref{pr:isodiamsimplex}};
\myarrowL{(\pmin-1.0,\qmin+1.15)}{(\pmin-0.05,\qmin+1.5)}
\def\sx{(\pmin-2}
\def\sy{\qmin+2.2}
\draw[simplexcolor!30, very thick, fill=simplexcolor!5] (\sx,\sy) -- (\sx-0.35,\sy-0.6) -- (\sx+0.35,\sy-0.6) -- cycle;
\filldraw[simplexpicturecolor] (\sx,\sy) circle (0.25ex);
\filldraw[simplexpicturecolor] (\sx-0.35,\sy-0.6) circle (0.25ex);
\filldraw[simplexpicturecolor] (\sx+0.35,\sy-0.6) circle (0.25ex);

\draw (1.4,\qmin-1.2)   node[above] {\autoref{conj:twopoint}};
\myarrowR{(1.4,\qmin-0.7)}{(-0.68,\qmin+2.7)}
\draw (-1.3,\qmin-1.15)   node[above] {\autoref{le:twopoint}};
\myarrowL{(-1.35,\qmin-0.77)}{(-1.2,\qmin-0.05)}
\draw (-3.35,\qmin-1.2)   node[above] {\autoref{th:2deqtriangle}(b)};
\myarrowL{(-3.35,\qmin-0.67) }{(\qmin+2,\qmin+0.85)}
\def\tpx{-5.75} 
\def\tpy{\qmin-0.9} 
\def\tpl{0.7} 
\draw[twoptcolorshading!45, very thick] (\tpx,\tpy) -- (\tpx+\tpl,\tpy);
\filldraw[twoptcolor] (\tpx,\tpy) circle (0.25ex);
\filldraw[twoptcolor] (\tpx+\tpl,\tpy) circle (0.25ex);

\draw (\n+1.5,\n-1.2)  node[above] {P\'{o}lya--Szeg\H{o}};
\draw (\n+1.5,\n-1.6)  node[above] {\autoref{conj:ps45}};
\myarrowR{(\n+0.45,\n-1)}{(\n-1.8,\n-1)}
\draw (-4,\n+0.7)  node[right] {\autoref{conj:riesz}};
\myarrowL{(-1.8,\n+0.7)}{(-0.8,\n-0.8)}
\draw (\pmin-1.7,\n-2+0.50) node[above] {P\'{o}lya--Szeg\H{o}};
\draw (\pmin-1.7,\n-2+0.1) node[above] {\cite[p.{\,}19]{PS51}};
\myarrowL{(\pmin-0.8,\n-1.5)}{(\pmin-0.1,\n-1.9)}
\draw (\pmin-1.7,\n-0.7) node[above] {\autoref{th:isodiametric}};
\myarrowL{(\pmin-0.77,\n-0.5)}{(\pmin-0.1,\n-0.9)}
\draw (\pmin-1.7,\n+0.4) node[above] {isodiametric};
\draw (\pmin-1.7,\n) node[above] {inequality};
\myarrowL{(\pmin-0.75,\n+0.45)}{(\pmin-0.1,\n+0.1)}
\draw (\n+1.3,\n+0.5)  node[above] {Watanabe};
\draw (\n+1.3,\n)  node[above] {\cite[Theorem 2]{W83}};
\myarrowR{(\n+0.4,\n+0.65)}{(\n-1,\n+0.1)}
\filldraw[ballcolor] (-0.7,\n+0.9) circle (1.6ex);

\draw (\n+1.65,-2.2)  node[above] {$+\infty$ };
\draw (\n+1.65,-2.6)  node[above] {\autoref{le:maxinfty}};
\myarrowR{(\n+1.22,-1.95)}{(\n-1,-1.7)}

\end{tikzpicture}

\caption{\label{fig:pqdiagram2D} ($n=2$) Maximizing $\capq(K)/\capp(K)$ for $K \subset \R^2$. Solid regions and solid segments indicate rigorous results. Striped regions and dashed curves are conjectural. Blue corresponds to the disk, red to the regular three-point set, and green to the two-point set.}
\end{figure}
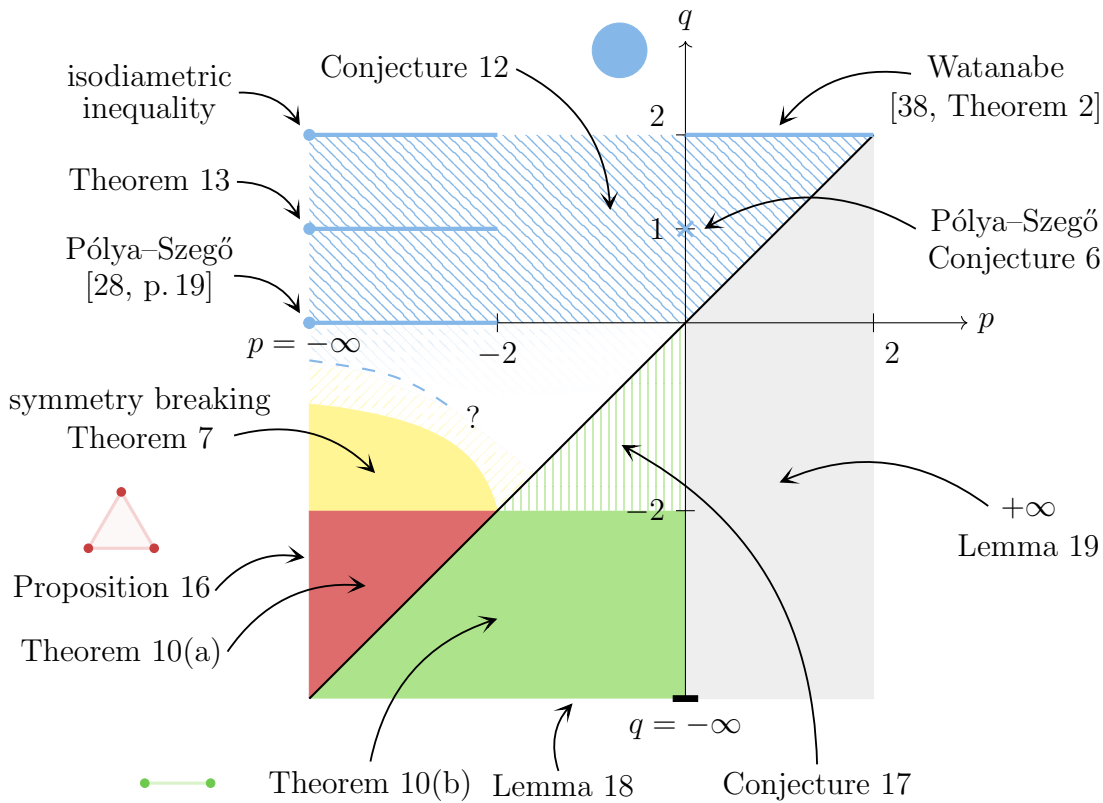



Motivating the entire upper region of \autoref{fig:pqdiagram2D} is an unresolved question by P\'{o}lya and Szeg\H{o}, marked at the point $(p,q)=(0,1)$. They conjectured that under the passage from logarithmic to Newtonian capacity, the disk retains more of its capacity than any other planar set, namely $2/\pi \simeq 64$\% of it. 
\begin{conjecture}[P\'{o}lya and Szeg\H{o} \protect{\cite[Conjecture (1.3)]{PS45}}] \label{conj:ps45}
If $K \subset \R^2$ is compact then 
\[
\capone(K) \leq \frac{2}{\pi} \capzero(K) ,
\]
\end{conjecture}
Equality holds when $K$ is a closed disk, since the unit disk has logarithmic capacity $\capzero(\overline{\B}^2) =1$ and Newtonian capacity $\capone(\overline{\B}^2) = 2/\pi$ (see formulas in \cite[Appendix A]{CL24b}). 

They restated the question in their 1951 book ``Isoperimetric Inequalities in Mathematical Physics'' \cite{PS51} (see page 10, and item {\#}17 on page 18). The question extends naturally to higher dimensions for $(p,q)=(n-2,n-1)$, as indicated later on \autoref{fig:pqdiagram}. See the examples in \autoref{sec:examples} that illustrate the conjecture. 

The best partial result toward \autoref{conj:ps45} is due to Szeg\H{o} \cite{S52}, \cite[pp.\,141--144]{S56}. He showed using conformal mapping techniques and the Dirichlet integral characterization of capacity that if $K$ is connected then  
\[
\capone(K) \leq (1.07) \frac{2}{\pi} \capzero(K) .
\]
Thus, Szeg\H{o} got within 7\% of the conjectured optimal value in the plane. 

Continuing into the third quadrant of \autoref{fig:pqdiagram2D}, where $p<q<0$, we believe a symmetry breaking threshold should exist, indicated by the dashed curve in the figure. Above the threshold the ball should maximize the capacity ratio while below it some other shape should be maximal. Finding the precise threshold seems challenging, but by comparing the ball with a regular three-point set we prove symmetry breaking for a substantial parameter region, plotted in \autoref{fig:pqdiagram2D}. The rightmost corner of the region is at $(p,q)=(-2,-2)$. In this parameter region, the maximizing set might perhaps be like an equilateral triangle with rounded sides.  
\begin{theorem}[Symmetry breaking] \label{th:symmetrybreaking2dim}
Let $q_* \simeq -0.856$ satisfy $2 \sqrt{\pi} \Gamma(1 - q_*/2) = 3 \Gamma((1 - q_*)/2)$. If 
\[
-\infty < p < \frac{q \log(4/3)}{\log \frac{2 \sqrt{\pi} \Gamma(1 - q/2)}{3 \Gamma((1 - q)/2)}} , \qquad -2<q<q_* ,
\]
then the ratio $\capq(K)/\capp(K)$ is less for $K$ a ball than when $K $ is a regular three-point set. 
\end{theorem}
The proof is in \hyperlink{proofof:th:symmetrybreaking2dim}{\autoref*{sec:threepoint}}. 

Taking now a slight detour, the capacity of an arbitrary three-point set (vertices of an arbitrary triangle) can be computed explicitly when $p<0$. 
\begin{theorem}[Capacity of a three-point set] \label{th:trianglecapformula} Let $p<0$ and $t=|p|>0$. If $T\subset\R^2$ is a three-point set with distances $a, b, c$ between its points, where $0<a,b \leq c$, then
\[
\capp(T) = \begin{cases}
2^{-1/t}c & \text{if\quad} a^{t}+b^{t}\leq c^{t}, \\
\dfrac{2^{1/t}abc}{\big( 4(ab)^{t}-(a^{t}+b^{t}-c^{t} )^2 \big)^{1/t}} & \text{if\quad} a^{t}+b^{t} > c^{t}.
\end{cases}
\]
The $p$-equilibrium measure on $T$ is unique. 
\end{theorem}
The proof in \hyperlink{proofof:th:trianglecapformula}{\autoref*{sec:threepoint}} reveals that in the first case, where $a^{t}+b^{t}\leq c^{t}$, the equilibrium measure concentrates at the endpoints of the longest side. In the other case, the measure lies on all three points. 

Capacities of some $k$-point sets, both regular and nonregular, are given in the same section. 

Continuing now with the restricted class of two- and three-point sets, we obtain the following sharp upper and lower bounds, proved in \autoref{sec:threeptratioproof}. 
\begin{theorem}[Optimal capacity ratios among two- and three-point sets]\label{th:threeptratio}
Assume $p<q<0$.  If $T\subset \R^2$ is a two-point or three-point set, then
\[
\left( \frac{1}{2} \right)^{\!\! \frac{1}{p}-\frac{1}{q}} \leq \frac{\capq(T)}{\capp(T)}\leq \left( \frac{2}{3} \right)^{\!\! \frac{1}{p}-\frac{1}{q}},
\]
with the minimum attained in particular by every two-point set and the maximum attained if and only if $T$ is a regular three-point set.
\end{theorem}

Building on the last result, we prove the regular three-point set conjecture and some of the two-point conjecture. That is, \autoref{conj:simplex} from the next section holds in the planar case and so does \autoref{conj:twopoint} for $q \leq-2$. 
\begin{theorem}[$p<0$ and $q\leq -2$] \label{th:2deqtriangle} Consider compact sets $K\subset \R^2$ containing more than one point. 

(a) If $p<q\leq-2$ then $\capq(K)/\capp(K)$ is maximal for $K$ a regular three-point set, with the maximal value being $(2/3)^{1/p-1/q}$. 

(b) If $q<p<0$ and $q\leq -2$, then $\capq(K)/\capp(K)$ is maximal for $K$ a two-point set, with maximal value $2^{1/q-1/p}$.
\end{theorem}
The proof of the theorem is in \autoref{sec:2dtriangleproof}. A key ingredient is a result of Björck \cite[Theorem 12]{B56} which states for $p<-2$ that the equilibrium measure of $K\subset \Rn$ is supported on at most $n+1$ extreme points of the convex hull of $K$. Thus in the planar case we may reduce to an arbitrary three-point set whose $p$-capacity we get from \autoref{th:trianglecapformula}, and that result with \autoref{th:threeptratio} helps prove \autoref{th:2deqtriangle}. 

 A word of caution is warranted, because the regular three-point set and the two-point set are not the only maximizers for \autoref{th:2deqtriangle}. For instance, when $p<q<-2$, any subset of the equilateral triangle that contains its vertices will have the same $p$ and $q$-capacities, and will therefore also maximize the ratio in part (a) of the theorem. Similarly, for part (b), a ball and almost-degenerate triangle will (for many parameters) have the same capacity and equilibrium measure as a two-point set, by \cite[Theorem 4.6.6]{BHS19} and \autoref{th:trianglecapformula}, and so can also be maximizing sets. 
 
 Although the maximizing sets differ in these examples, the underlying equilibrium measures are admittedly the same. Surprisingly, even the measures can differ. In \autoref{sec:2dtriangleproof} we construct a maximizing set for the capacity ratio when $q<p<-2$ that supports two different $p$-equilibrium measures. One of those equilibrium measures is supported not on a $2$-point subset of the maximizer but rather on a $3$-point subset. 

\section{\bf Higher dimensions: results and open problems}
\label{sec:conjectures}


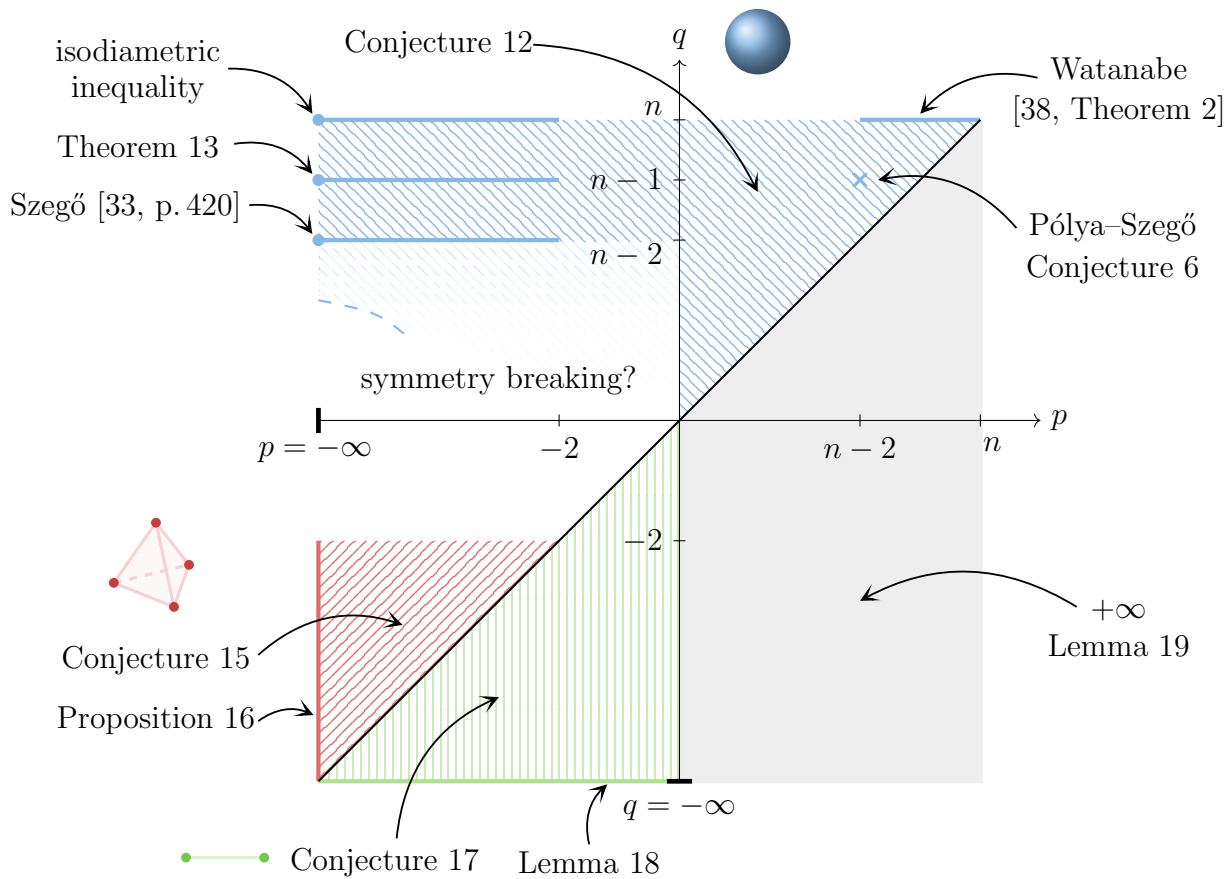
\begin{figure}
\begin{tikzpicture}[scale=0.8]

\def\n{5}
\def\pmax{6}
\def\pmin{-6}
\def\qmax{6}
\def\qmin{-6}


\filldraw[top color=ballcolor!40, bottom color=white, middle color=ballcolor!5, draw=none] (\pmin+0.02,\n-3) -- (\pmin+0.02,\n-2.02) -- (0,\n-2.02) -- (0,0);
\filldraw[pattern={Lines[angle=-45,distance=3pt,line width=2.4pt, xshift=2pt]}, pattern color=white, draw=none] (\pmin+0.02,\n-3) -- (\pmin-0.2,\n-2) -- (0,\n-2) -- (0,0);
\filldraw[pattern={Lines[angle=-45,distance=3pt,line width=2.4pt, xshift=2pt]}, pattern color=white, draw=none] (\pmin+0.02,\n-3) -- (\pmin-.2,\n-2) -- (0,\n-2) -- (0,0);
\draw[white, ultra thick] (\pmin,\n-3) -- (0,0);
\filldraw[pattern={Lines[angle=-45,distance=3pt, line width=0.5pt]}, pattern color=ballcolor, draw=none] (0,0) -- (\n, \n) -- (\pmin, \n) -- (\pmin, \n-2.02) -- (0, \n-2.02) ;
\filldraw[color=mygray, draw=mygray, ultra thick] (\n,\n) -- (\n,\qmin) -- (0,\qmin) -- (0,0);
\filldraw[pattern={Lines[angle=90,distance=3pt, line width=0.5pt]}, pattern color=twoptcolorshading, draw=none] (\pmin,\qmin) -- (0, \qmin) -- (0, 0)  ;
\filldraw[pattern={Lines[angle=45,distance=3pt, line width=0.5pt]}, pattern color=simplexcolor, draw=none] (\pmin,\qmin) -- (\qmin, -2) -- (-2, -2)  ;


\draw[-, ballcolor, ultra thick] (\n-2,\n) to (\n,\n);
\draw[-, ballcolor, ultra thick] (\pmin,\n) to (-2,\n);
\draw[-, ballcolor, ultra thick] (\pmin,\n-1) to (-2,\n-1);
\draw[-, ballcolor, ultra thick] (\pmin,\n-2) to (-2,\n-2);
\draw[ballcolor,fill=ballcolor] (\pmin,\n) circle (.5ex);
\draw[ballcolor,fill=ballcolor] (\pmin,\n-2) circle (.5ex);
\draw[ballcolor,fill=ballcolor] (\pmin,\n-1) circle (.5ex);
\draw[-, ultra thick, simplexcolor] (\pmin,\qmin) -- (\pmin, -2);
\draw[twoptcolorshading, ultra thick] (0, \qmin ) -- (\pmin,\qmin);
\draw[ballcolor, dash pattern=on 5pt off 5.1pt, thick] (\pmin,\n-3)  .. controls (\pmin+1,\n-3.2)  .. (\pmin+1.59,\n-3.7);
\draw[ballcolor, very thick] (\n-2+0.11,\n-1+0.11) -- (\n-2-0.11,\n-1-0.11) ;
\draw[ballcolor, very thick] (\n-2-0.11,\n-1+0.11) -- (\n-2+0.11,\n-1-0.11) ;


\draw (\n+2.3,-3.5)  node[above] {$+\infty$};
\draw (\n+2.3,-4.1)  node[above] {\autoref{le:maxinfty} };
\myarrowR{(\n+1.6,-3.1)}{(\n-2,-3)}

\draw (-3,0.25)  node[above] {symmetry breaking?};

\draw (\n+2.2,\n-2.2)  node[above] {P\'{o}lya--Szeg\H{o}};
\draw (\n+2.2,\n-2.9)  node[above] {\autoref{conj:ps45}};
\myarrowR{(\n+0.6,\n-1.8)}{(\n-1.7,\n-1)}

\draw (-4,\n+0.85)  node[above] { \autoref{conj:riesz}};
\myarrowL{(-2.43,\n+1.25)}{(\n-3.7,\n-1.2)}
\draw (\pmin-3,\n+0.8) node[above] {isodiametric};
\draw (\pmin-3,\n+0.1) node[above] {inequality};
\myarrowL{(\pmin-1.5,\n+0.85)}{(\pmin-0.1,\n+0.1)}
\draw (\pmin-3,\n-0.8) node[above] {\autoref{th:isodiametric}};
\myarrowL{(\pmin-1.4,\n-0.4) }{(\pmin-0.1,\n-1+0.1)}
\draw (\pmin-3.2,\n-2+0.1) node[above] {Szeg\H{o} \cite[p.{\,}420]{S31}};
\myarrowL{(\pmin-1.15,\n-1.42)}{(\pmin-0.1,\n-2+0.1)}
\draw (\n+2.3,\n+0.5)  node[above] {Watanabe};
\draw (\n+2.3,\n-0.3)  node[above] {\cite[Theorem 2]{W83}};
\myarrowR{(\n+0.95,\n+0.7)}{(\n-1,\n+0.1)}
\shade[ball color=ballcolor] (1.3,\n+1.3) circle (3ex);

\draw (-4.9,\qmin-1.76)   node[above] {\autoref{conj:twopoint}};
\myarrowL{(-4.7,\qmin-0.95) }{(-3,\qmin+1.4)}
\draw (-1.5,\qmin-1.7)   node[above] {\autoref{le:twopoint}};
\myarrowL{(-1.5,\qmin-1.05)}{(-1.2,\qmin-0.05)}
\def\tpx{\pmin-2.2} 
\def\tpy{\qmin-1.27} 
\def\tpl{1.3} 
\draw[twoptcolorshading!45, very thick] (\tpx,\tpy) -- (\tpx+\tpl,\tpy);
\filldraw[twoptcolor] (\tpx,\tpy) circle (0.4ex);
\filldraw[twoptcolor] (\tpx+\tpl,\tpy) circle (0.4ex);

\draw (\pmin-2.7,\qmin+1.6)   node[above] {\autoref{conj:simplex}};
\myarrowL{(\pmin-1,\qmin+2.1)}{(\pmin+1.4,\qmin+2.6)}
\draw (\pmin-2.7,\qmin+0.6)   node[above] {\autoref{pr:isodiamsimplex}};
\myarrowL{(\pmin-1,\qmin+1)}{(\pmin-0.08,\qmin+1.1)}
\def\sx{(\pmin-2.7} 
\def\sy{-1.7} 
\draw[simplexcolor!30, very thick, fill=simplexcolor!5] (\sx,\sy) -- (\sx-0.7,\sy-1) -- (\sx+0.3,\sy-1.4)  -- (\sx+0.55,\sy-0.7) -- cycle;
\draw[simplexcolor!30, very thick] (\sx,\sy) -- (\sx+0.3,\sy-1.4);
\draw[simplexcolor!30, dash pattern=on 3pt off 3pt, very thick] (\sx-0.7,\sy-1)  -- (\sx+0.55,\sy-0.7);
\filldraw[simplexpicturecolor] (\sx,\sy) circle (0.4ex);
\filldraw[simplexpicturecolor] (\sx-0.7,\sy-1)  circle (0.4ex);
\filldraw[simplexpicturecolor] (\sx+0.55,\sy-0.7) circle (0.4ex);
\filldraw[simplexpicturecolor] (\sx+0.3,\sy-1.4) circle (0.4ex);


\draw[black, thick]  (\qmin,\pmin) to (\n,\n);

\draw[->] (\pmin,0) to (\pmax,0);
\draw (\pmax,0)  node[right] {$p$};
\draw (\pmin-0.05,-0.1)  node[below] {$p=-\infty$};
\draw[black,fill=black] (\pmin-0.03,-0.2)  rectangle  (\pmin+0.03,0.2) ;

\draw[->] (0,\qmin) to (0,\qmax);
\draw (0,\qmax)  node[above] {$q$};
\draw (0,\qmin-0.05)  node[below] {$q=-\infty$};
\draw[black,fill=black] (-0.2,\qmin-0.03)  rectangle  (0.2,\qmin+0.03) ;

\draw[-] (\n,-0.1) to (\n,0.1);
\draw (\n+0.2, -0.1) node[below] {$n$};
\draw[-] (\n-2,-0.1) to (\n-2,0.1);
\draw (\n-2, -0.1) node[below] {$n-2$};
\draw[-] (-2,-0.1) to (-2,0.1);
\draw (-2, -0.1) node[below] {$-2$};

\draw[-] (-0.1,-2) to (0.1,-2);
\draw (-0.1,-2) node[left] {$-2$};
\draw[-] (-0.1,\n-2) to (0.1,\n-2);
\draw (-0.1,\n-2-0.25) node[left] {$n-2$};
\draw[-] (-0.1,\n-1) to (0.1,\n-1);
\draw (-0.1,\n-1) node[left] {$n-1$};
\draw[-] (-0.1,\n) to (0.1,\n);
\draw (-0.1,\n+0.2) node[left] {$n$};

\end{tikzpicture}
\caption{\label{fig:pqdiagram} ($n \geq 3$) Maximizing $\capq(K)/\capp(K)$ for $K \subset \Rn$. Solid regions of the diagram and the  solid segments indicate rigorous results. Striped regions and dashed curves are conjectural. Blue corresponds to the ball, red to the regular $(n+1)$-point set, and green to the two-point set. For dimensions $n=1$ and $n=2$, see the additional results in \autoref{fig:pqdiagram1D} and \autoref{fig:pqdiagram2D}.}
\end{figure}


In all dimensions $n \geq 1$, we start by showing the ratio of Riesz capacities achieves a maximum, when $p<0$. 
\begin{theorem}[Maximizer exists when $p<0$] \label{th:maximizer}
Let $n \geq 1$ and $q<n$.

If $p<0$ then the ratio $\capq(K)/\capp(K)$ has a maximizer among the family of compact sets $K \subset \Rn$ with positive $p$-capacity. 

If $p<0$ and $q<-2$ then the maximizer can be taken to be a finite set consisting of between $2$ and $n+1$ points, each of which is an extreme point for the convex hull of the set. 

If $p<-1$ then the maximizer can be taken to be convex. 
\end{theorem}

The proof is in \hyperlink{proofof:th:maximizer}{\autoref*{sec:additionalproofs}}. The same proof shows that maximizers also exist for $p<0$ in the limiting case $q=n$ where the ratio is $\Vol_n(K)^{1/n}/\capp(K)$. 
We do not know how to prove existence of a maximizer in the remaining case of interest, when $0 \leq p < q < n$. 

Knowing a maximizer exists, one naturally wants to discover its shape. For which parameter values might the maximizer be a ball? The scope of our next conjecture is broad: it unifies the long-standing open problem due to P\'{o}lya and Szeg\H{o} for logarithmic and Newtonian capacity with classical results of Carleman and Szeg\H{o} that minimize logarithmic and Newtonian capacity under a volume constraint, and Watanabe's generalization to Riesz capacity in the $\alpha$-stable case. Isodiametric inequalities appear too, as boundary cases of the conjecture. See \autoref{fig:pqdiagram}. 
\begin{conjecture}[Balls] \label{conj:riesz}
Suppose $n \geq 1$ and $p<q$. If either $0 \leq p < q < n$ or else $p<0$ and $n-2 \leq q < n$ then among compact $K \subset \Rn$ with positive $p$-capacity, the closed $n$-ball maximizes 
\[
\frac{\capq(K)}{\capp(K)} \qquad \text{and} \qquad \frac{\Vol_n(K)^{1/n}}{\capp(K)} .
\]
\end{conjecture}
The conclusion for volume is a limiting case obtained by letting $q \to n$ and calling on \autoref{th:Rieszmonotonicity}(c). As explained below, that second conclusion (which is equivalent to minimizing capacity among sets with given volume) is known to hold in the ``$\alpha$-stable process case'' $n-2 \leq p < n$ by work of Watanabe. 

The question of maximizing the ratio of Riesz capacities with $0<p<q<n$ was raised by Szeg\H{o} \cite[p.{\,}140]{S56}, who wrote ``Nothing is known about this more general, possibly very difficult question.'' Presumably he thought the ball a likely candidate for the maximizer. Note he did not examine the limiting case $q \to n$, and so did not connect the capacity ratio conjecture to minimization of capacity under an area or volume constraint. 

The conjecture is confirmed for some special families of sets by the examples in \autoref{sec:examples}.

\subsubsection*{Edge cases of \autoref{conj:riesz} with fixed volume ($q=n$)} \ 

1. Riesz capacity is minimal for a ball, given fixed volume, in the ``$\alpha$-stable processes'' regime: when $n \geq 2$ and $n-2<p<n$, the conjecture says the Riesz $p$-capacity is minimal for the $n$-ball of the same volume; this assertion was proved by Watanabe \cite[p.{\,}489]{W83}. See also Betsakos \cite{B04a,B04b} and M\'{e}ndez--Hern\'{a}ndez \cite{MH06}. The corresponding statement for $0<p<n-2$ is a long-standing open problem. 

2. The endpoint case of Watanabe's theorem says that Newtonian $(n-2)$-capacity is minimal for a ball given fixed volume, which is known by an old result of Szeg\H{o} \cite{PS51}, \cite[Theorem 5.12]{B19}, for $n \geq 3$ and $p=n-2$. Similarly, for $n=2$ and $p=0$, logarithmic capacity is minimal for a disk of the same area, by work of Carleman and later Szeg\H{o} \cite{PS51}. 

3. The midpoint case of Watanabe's theorem ($p=n-1,q=n$) says in the planar case ($p=1,n=2$) that Newtonian capacity is minimal for a disk among planar sets of given area. That is a result of Pólya and Szeg\H{o} \cite[inequality (1.2)]{PS45}, obtained by symmetrization in the horizontal directions. That is, 
\[
\capone(K) \geq \frac{2}{\pi} \sqrt{\frac{\Area(K)}{\pi}} ,
\]
with equality when $K \subset \R^2$ is a disk. 

Rearrangement and polarization methods underlie these known cases of the conjecture, along with probabilistic and Dirichlet integral characterizations of the relevant capacities. 

These known cases all have $n-2 \leq p < q = n$. The conjecture for $(p,q)=(n-2,n-1)$ was raised by P\'{o}lya and Szeg\H{o} (see \autoref{conj:ps45} earlier). Evidence for \autoref{conj:riesz} is weak when $p<n-2$, which means in particular that the logarithmic case of the conjecture ($p=0$) is not well supported when $n \geq 3$. 

\smallskip
Now we examine the left side of \autoref{fig:pqdiagram}, where by \autoref{th:Rieszmonotonicity}(b), as $p \to -\infty$ the $p$-capacity reduces to diameter. 
%
\begin{theorem}[Isodiametric theorem for capacity and volume] \label{th:isodiametric}
Suppose $n \geq 2$. If $q=n-2$ or $q=n-1$ then among compact $K \subset \Rn$ containing more than one point, the closed $n$-ball maximizes 
\[
\frac{\capq(K)}{\diam(K)} \qquad \text{and} \qquad \frac{\capq(K)}{\capp(K)} \quad \text{for} \ -\infty < p \leq -2 .
\]
When $q=n$, the analogous statements are that the ball maximizes $\Vol_n(K)^{1/n}/\diam(K)$ and $\Vol_n(K)^{1/n}/\capp(K)$ for $-\infty < p \leq -2$.
\end{theorem}
The statement about $\Vol_n(K)^{1/n}/\diam(K)$ is simply the classical isodiametric inequality. We have not seen the case $q=n-1$ stated in the literature. The statement when $q=n-2$ about $\capntwo(K)/\diam(K)$ is due to Szeg\H{o} \cite{S31} for $n=3$, by an argument that extends to all dimensions $n\geq 3$. 

The theorem is proved in \autoref{sec:isodiametric}, using known Brunn--Minkowski type inequalities for Riesz capacity. Note we do not consider $n=1$ in the theorem since that was handled already by \autoref{pr:onedim_upperleft}. 

We believe the last theorem should hold for all $q \in [n-2,n)$, and maybe for an even larger range of $q$-values. 
\begin{conjecture}[Isodiametric conjecture for Riesz capacity]
Let $n \geq 2$. A number $q(n) \leq n-2$ exists such that among compact sets $K \subset \Rn$, the ratio $\capq(K)/\diam(K)$ is maximal  for the ball when $q \in (q(n),n)$. 
\end{conjecture}
If the conjecture holds then $q(n) \to \infty$ as $n \to \infty$, because Burchard, Choksi and Hess-Childs \cite{BCH20} proved that for each $q>0$, the ball in $\Rn$ fails to maximize $\capq(K)/\diam(K)$ when $n$ is sufficiently large. 

Just as in the planar case treated in the previous section, we expect that a symmetry breaking threshold exists for the capacity ratio, indicated by the conjectural dashed curve in \autoref{fig:pqdiagram} with left endpoint at height $q(n)$. Above the threshold the ball would provide the maximizer, while some other shape would maximize below the threshold. To get an estimate on the location of the threshold one could compare the ball with the regular $(n+1)$-point set, like we did for $2$ dimensions in \autoref{th:symmetrybreaking2dim}. But it turns out that this comparison provides symmetry breaking only for a subregion of the third quadrant, where $q<0$, whereas when $n \geq 3$ one expects symmetry breaking to occur in parts of the second quadrant, where $q>0$. For example, a particle simulation by Burchard, Choksi and Hess-Childs \cite[Figure 3]{BCH20} in $3$ dimensions suggests that $q(3) \geq 0.01$ and that the maximizer in that case looks like the edge set of a rounded, regular tetrahedral shape. Further numerical investigations with $p>-\infty$ would be most welcome, in order to expand the second quadrant region where symmetry is known to break. 

When $(p,q)$ lies sufficiently below the symmetry breaking transition region, we believe the vertices of a regular simplex should form a maximizing set. 
\begin{conjecture}[Regular $(n+1)$-point set conjecture] \label{conj:simplex}
Suppose $n \geq 3$ and $p < q \leq -2$. Among compact sets $K \subset \Rn$ containing more than one point, the regular $(n+1)$-point set maximizes $\capq(K)/\capp(K)$.
\end{conjecture}
The conjecture holds for $n=2$ by \autoref{th:2deqtriangle}(a), and for $n=1$ it is immediate from \autoref{le:onedim_lowerleft}. The edge case of the conjecture in the next proposition ($p=-\infty$, diameter) is due to Lim and McCann \cite[Corollary 2.2]{LM21}. Our proof in \hyperlink{proofof:pr:isodiamsimplex}{\autoref*{sec:additionalproofs}} is more direct, building on Bj\"{o}rck's result \cite{B56} that when $q<-2$, the equilibrium measure is supported on an $(n+1)$-point set. Our method gets the equality statement only when $q<-2$, though, and so we still rely on Lim and McCann for the equality case when $q=-2$.  
\begin{proposition}[Isodiametric inequality for capacity when $q \leq -2$] \label{pr:isodiamsimplex} Suppose $n\geq 2$ and $q \leq -2$. 

Among compact sets $K \subset \Rn$ containing more than one point, the regular $(n+1)$-point set maximizes the ratio $\capq(K)/\diam(K)$. 

Further, regular $(n+1)$-point sets are essentially the only maximizers: if $K$ is a maximizing set then every $q$-equilibrium measure on $K$ is a uniform probability measure on some regular $(n+1)$-point subset that possesses the same $q$-capacity and diameter as $K$. 
\end{proposition}
Next, on the diagonal in \autoref{fig:pqdiagram} where $p=q$, the capacity ratio simply equals $1$. Crossing below the diagonal, we enter the region where $q<p$ and arrive at a new conjecture. 
\begin{conjecture}[Two-point conjecture] \label{conj:twopoint}
Suppose $n \geq 2$ and $q < p < 0$. Among compact sets $K \subset \Rn$ containing more than one point, any set with  exactly two points maximizes $\capq(K)/\capp(K)$.
\end{conjecture}
Two-point sets are not the only potential maximizing sets. For example, a line segment has its equilibrium measure supported at just two points, when $p \leq -1$, and so does a ball when $p < -2$; see \cite[Proposition 4.6.1 and Theorem 4.6.6]{BHS19}. Thus  different sets can both be maximizers, albeit in this example with the same underlying equilibrium measure. 

Note that one may restrict attention to finite sets with at most $n+1$ points in \autoref{conj:simplex} and in \autoref{conj:twopoint} when $q<-2$, in view of \autoref{th:maximizer}. 

In $2$ dimensions, \autoref{th:2deqtriangle}(b) largely proves \autoref{conj:twopoint}. In all dimensions, the edge case $q = -\infty$ of the conjecture is established by the following lemma, which we prove in \hyperlink{proofof:le:twopoint}{\autoref*{sec:additionalproofs}}.

\begin{lemma}[Minimizing capacity with given diameter] \label{le:twopoint}
Suppose $n \geq 1$ and $p<0$. Among compact sets $K \subset \Rn$ containing more than one point, each set with  exactly two points maximizes 
\[
\frac{\diam(K)}{\capp(K)} .
\]
\end{lemma}

In the remaining region of \autoref{fig:pqdiagram}, the capacity ratio is unbounded according to the next lemma, proved in \hyperlink{proofof:le:maxinfty}{\autoref*{sec:additionalproofs}}.
\begin{lemma}[No maximum] \label{le:maxinfty}
If $0 \leq p < n$ and $q<p$ then a compact set $K \subset \Rn$ exists for which $\capq(K)$ and $\diam(K)$ are positive while $\capp(K)$ and $\Vol_n(K)$ are zero, so that 
\[
\frac{\capq(K)}{\capp(K)} = \infty 
\]
for this set, and similarly in the ``bottom edge'' case where $\capq$ is replaced by $\diam$ and the ``right edge'' case where $\capp$ is replaced by $\Vol^{1/n}$.
\end{lemma}
Many such sets exist. For example, when $(p,q,n)=(1,0,2)$ we could take $K$ to be a line segment in the plane, which has positive   $0$-capacity and diameter along with vanishing Newtonian $1$-capacity and area.

\section{\bf Examples for \autoref{conj:ps45} and \autoref{conj:riesz}}
\label{sec:examples}

Explicitly computable capacity examples are plentiful in the planar logarithmic case thanks to the connection with conformal mapping. For Newtonian capacity only a few examples are computable, and even fewer for Riesz capacity. 

Formulas for the Riesz capacity of the unit ball in all dimensions are collected in \cite[Appendix A]{CL24b}. Notice that our definition of Riesz capacity takes the $p$-th root of the energy. In this choice we follow Hayman and Kennedy \cite{HK76}, while  other authors such as Landkof \cite{L72} and Borodachov, Hardin and Saff \cite{BHS19} do not take the root. Landkof also multiplies by a constant factor. Thus when consulting those sources, one must remember to modify their formulas accordingly. 

Newtonian capacities of line segments, disks, balls, ellipsoids and some other sets are computed in Hayman \cite[{\S\S}7.3.1, 7.3.2]{H89}, Landkof \cite[pp.{\,}165, 172]{L72}, Rumely \cite[{\S}5.2]{R89}, and Szeg\H{o} \cite[{\S\S}12--17]{S45}. Note that Landkof's kernel for Newtonian energy in $n=3$ dimensions is $\pi/|x|$ whereas ours is $1/|x|$, and so his $1$-capacity values should be multiplied by $\pi$. 

\subsection*{Elliptical examples for \autoref{conj:ps45}} The ratio $\capone(K)/\capzero(K)$ in \autoref{conj:ps45} equals $0$ for a segment and $2/\pi$ for a disk. We may interpolate between these two extremes with a family of filled-in ellipses having semi-axes $a=1$ and $b \in [0,1]$. The logarithmic capacity of this ellipse is $(1+b)/2$ and its $1$-capacity equals $1/F(\sqrt{1-b^2})$ where $F$ is the complete elliptic integral of the first kind with input parameter $k=\sqrt{1-b^2}$: for this last fact see Hayman \cite[Example 7.21, Theorem 7.14]{H89} (take $m=3$ and semiaxes $1,b,0$) or Landkof \cite[pp.{\,}165, 172]{L72} or Szeg\H{o} \cite[{\S}12]{S45}. Note too that Tee \cite[{\S\S}8-10]{T05} makes interesting observations. 

Thus the ratio of $1$-capacity over logarithmic capacity for the ellipse is  
\[
\frac{2}{(1+b) F(\sqrt{1-b^2})} .
\]
This ratio increases from the line segment value $0$ to the disk value $2/\pi$ as $b$ increases from $0$ to $1$, as one may confirm graphically. Hence \autoref{conj:ps45} holds for the family of filled-in ellipses. 

\subsection*{Higher dimensional ellipsoidal examples for \autoref{conj:riesz}}
Take $p=n-2$ and $q=n-1$. The $(n-1)$-capacity of the closed unit ball in $\Rn, n \geq 2$, is 
\[
\left( \frac{\Gamma(n/2)}{\Gamma(1/2) \Gamma((n+1)/2)} \right)^{\! \! 1/(n-1)}
\]
and the Newtonian $(n-2)$-capacity equals $1$; see \cite[Appendix A]{CL24b} for these formulas. Hence the last displayed expression is exactly the conjectured maximum value for the capacity ratio in \autoref{conj:riesz}, when $(p,q)=(n-2,n-1)$. For example, substituting $n=2$ gives $2/\pi$ while $n=3$ gives $1/\sqrt{2}$ for the conjectured maximum value. 

When $n=3$, we can gain insight into the conjecture by considering the family of solid $3$-dimensional ellipsoids $E(b)$ with semi-axes $1,1,b$, where $b \in [0,1]$, so that the ellipsoids interpolate between the $2$-ball ($b=0$) and the $3$-ball ($b=1$). The $1$-capacity of the ellipsoid is 
\[
\capone(E(b)) = \frac{\sqrt{1-b^2}}{\arcsin \sqrt{1-b^2}} ,
\]
by evaluating \cite[Theorem 7.14]{H89} with dimension $m=3$ and semiaxes $1,1,b$; note this formula yields the correct endpoint values $2/\pi$ for the $2$-ball as $b \to 0$ and $1$ for the $3$-ball as $b \to 1$. The $2$-capacity of the ellipsoid is 
\[
\captwo(E(b)) = \left( \frac{\sqrt{1-b^2}}{2 \arcsinh \sqrt{b^{-2}-1}} \right)^{\! \! 1/2} ,
\]
this time by evaluating \cite[Theorem 7.14]{H89} with dimension $m=4$ and semiaxes $1,1,b,0$. Notice $\captwo(E(b)) \to 0$ as $b \to 0$, meaning that a $2$-ball has $2$-capacity equal to $0$, and $\captwo(E(b)) \to 1/\sqrt{2}$ as $b \to 1$, which correctly gives  the $2$-capacity of the $3$-ball. 

The ratio $\captwo(E(b))/\capone(E(b))$ can be plotted numerically. One finds that it increases from the $2$-ball value $0$ to the $3$-ball value $1/\sqrt{2}$ as $b$ increases from $0$ to $1$. Thus \autoref{conj:riesz} is confirmed within this family of ellipsoids, for $n=3$ with $(p,q)=(1,2)$. 

\subsection*{Examples where the dimension varies} In any given dimension, the ball and sphere are the only sets for which a capacity formula is known to us for all $p$, although when $p<0$ one does have formulas for arbitrary three-point sets and regular $k$-point sets by \autoref{th:trianglecapformula} and \autoref{sec:threepoint}. 

Taking a different perspective, let us regard the ambient dimension as variable. If \autoref{conj:riesz} is correct then balls of higher dimension should give larger values for the ratio, since the $n$-ball can be regarded as a set in $\Rnp$ and there the $(n+1)$-ball should be the maximizer. That is, if one fixes $0<p<q$ and plots $\capq(\overline{\B}^n)/\capp(\overline{\B}^n)$ as a function of $n \in [q,\infty)$ then the resulting curve should be  increasing. This monotonicity holds true for every choice of parameters we have examined. See \autoref{fig:npq} for a plot with $p=2,q=8$ and $n \geq 8$. 
\begin{figure}
\begin{center}
\includegraphics[width=0.5\textwidth]{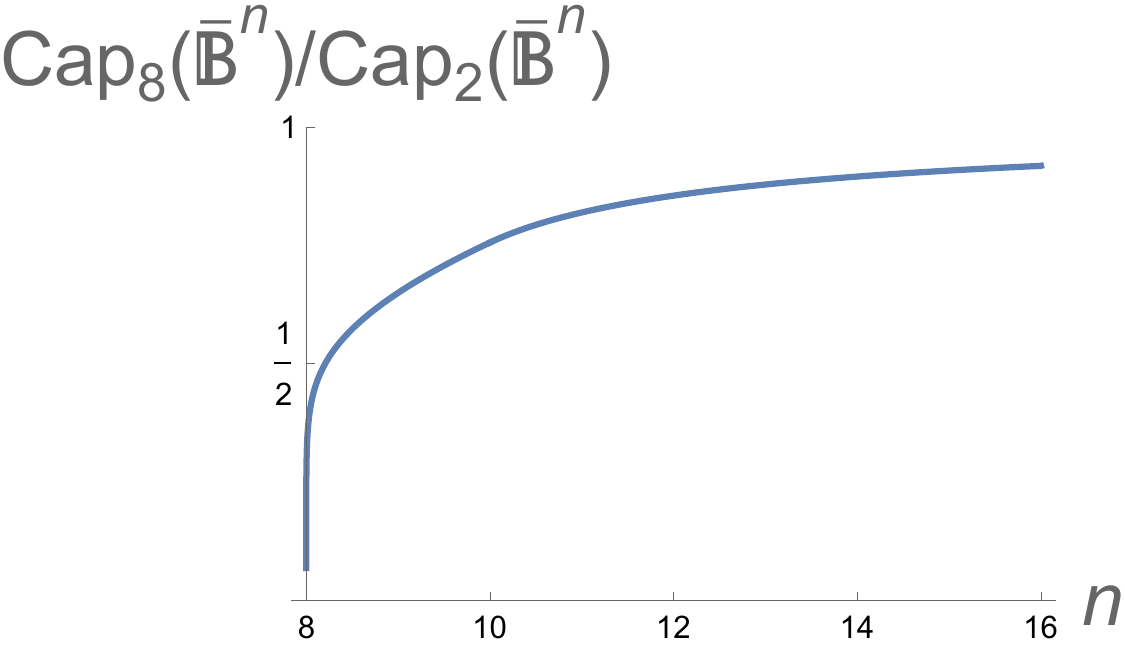}
\end{center}
\caption{\label{fig:npq} Plot of $\capq(\overline{\B}^n)/\capp(\overline{\B}^n)$ as a function of $n$, for $p=2$ and $q=8$. The ratio increases with $n$, in accordance with \autoref{conj:riesz}.}
\end{figure}

\section{\bf Proof of \autoref{th:onedim}: minimizing capacity for given length, in $1$ dimension}
\label{sec:onedimproof}

If $K$ has length zero then there is nothing to prove. 

Suppose from now on that $\Length(K)>0$ and $p<1$, which implies $\capp(K)>0$ (for example, by \autoref{th:Rieszmonotonicity}). After rescaling we may suppose $\Length(K)=2$, and by translation we may suppose the leftmost point of $K$ sits at $x=0$. Write 
\[
I=[0,2]
\]
for the translation of $\overline{\B}^1$ by $1$ unit to the right. We first show 
\[
\capp(K) \geq \capp(I) 
\] 
and then later deal with the equality statement in the theorem. 

A mapping of $[0,\infty)$ onto $I$ is given by $f(x)=\Length(K \cap [0,x])$, noting that $f$ measures the amount of the set $K$ lying to the left of $x$. Clearly $f$ is increasing, continuous, and a contraction, since  
\[
0 \leq x \leq y \qquad \Longrightarrow \qquad 0 \leq f(y)-f(x) \leq y-x .
\]
Notice $f(x)=2$ for all sufficiently large $x$, since $K$ is bounded and has measure $2$. We could now conclude $\capp(K) \geq \capp(I)$ by the known result, valid in all dimensions, that capacity decreases under contraction \cite[p.{\,}158]{L72}. But in the current $1$-dimensional case, it is more helpful for the equality case later to argue directly as follows. 

We will construct a right-inverse of $f$ that is expanding. For $t \in I$, define 
\begin{align*}
g(t) 
& = \min \{ x \geq 0 : f(x)=t \} \\
& = \text{leftmost point of $f^{-1}(\{t\})$},
\end{align*}
so that $f(g(t))=t$ by definition and hence $g$ is a right-inverse for $f$. Clearly $g(0)=0$, and $g$ is strictly increasing and hence Borel measurable. 

We claim $g$ maps $I$ into $K$. This claim is clear for $t=0$, since $0 \in K$. So suppose $0<t \leq 2$. Let $x=g(t)$, so that $x>0$ and $f(x)=t$. If $x \notin K$ then some interval immediately to the left of $x$ would lie outside $K$ and so $f$ would equal $t$ on that whole interval, implying $g(t)<x$, a contradiction. Therefore $x \in K$, as wanted. 

Further, $g$ is expanding:  
\begin{equation} \label{eq:expanding}
0 \leq s \leq t \leq 2 \qquad \Longrightarrow \qquad t-s \leq g(t)-g(s) 
\end{equation}
since $t-s=f(g(t))-f(g(s)) \leq g(t)-g(s)$ by the contraction property of $f$. 

We may push measures forward from $I$ onto $K$ via $g$: given a probability measure $\nu$ on $I$, let $\mu = \nu \circ g^{-1}$ be its pushforward by $g$, so that $\mu$ is a probability measure on $K$.  If $0<p<1$ then 
\begin{align}
V_p(K) 
& \leq \int_K \int_K |x-y|^{-p} \, d\mu d\mu \notag \\
& = \int_I \int_I |g(s)-g(t)|^{-p} \, d\nu d\nu \qquad \text{since $\mu = \nu \circ g^{-1}$} \notag \\
& \leq \int_I \int_I |s-t|^{-p} \, d\nu d\nu \label{eq:gbound}
\end{align}
because $g$ is expanding. Taking the minimum over all probability measures $\nu$, we conclude that $V_p(K) \leq V_p(I)$ and hence $\capp(K) \geq \capp(I)$. 

If $p=0$, the argument is the same except using the logarithmic kernel. If $p<0$ then the argument is essentially the same, except with the direction of each inequality reversed and with a maximum  instead of a minimum taken over the measures $\nu$. 

\subsection*{Equality statements} 
Take $\nu$ in the argument above to be the equilibrium measure on the interval $I$, so that the energy in \eqref{eq:gbound} equals $V_p(I)$. (The explicit formula for the equilibrium measure of an interval is recalled in \cite[Appendix A]{CL24b}.) Note that $\nu([0,\e])>0$ for each $\e>0$, since otherwise the energy could be driven strictly lower (when $p \geq 0$) or higher (when $p<0$) by stretching out $\nu$ to the full interval. Similarly $\nu([2-\e,2])>0$ for each $\e>0$.

\smallskip
``$\Longrightarrow$'' Suppose equality holds in the theorem, so that $\capp(K) = \capp(I)$ in our proof above. Equality must hold in \eqref{eq:gbound} and so $|g(s)-g(t)|=|s-t|$ for $(\nu \times \nu)$-almost every $(s,t) \in I \times I$. In particular, this equation holds for some points $s_j \searrow 0$ and $t_j \nearrow 2$, so that equality holds in \eqref{eq:expanding} for $s_j$ and $t_j$. The expanding property of $g$ then forces equality to hold in \eqref{eq:expanding} for all intermediate points $s,t$, with $0 \leq s_j<s<t<t_j \leq 2$. Taking $j \to \infty$ we obtain that $g(t)-g(s)=t-s$ whenever $0 <s<t<2$. Let $s$ tend to $0$ and write $ \lim_{s \searrow 0} g(s) = x_*$ so that $g(t)=x_*+t$ when $t \in (0,2)$. Since $g$ maps $I$ into $K$, we see $x_*+(0,2) \subset K$. Taking the closure, the nondegenerate interval $J=x_*+I$ is contained in $K$. Let $Z = K \setminus J$. 

If $p \leq -1$ then because capacity equals $2^{1/p}$ times diameter by \autoref{th:Rieszmonotonicity}(b), we have 
\[
2^{1/p} \diam(K)=\capp(K)=\capp(I)=\capp(J)=2^{1/p} \diam(J) .
\]
Thus $K$ and $J$ have the same diameter, and so $K$ cannot contain any points in addition to $J$. That is, $Z$ is empty and $K$ equals the nondegenerate interval $J$. 

Now suppose $-1<p<1$. The equilibrium measure of $K$ is certainly unique, since $p>-2$. Because $K$ and its subset $J=x_*+I$ have the same capacity, they must have the same equilibrium measure, which is supported in $J$. Call that equilibrium measure $\gamma$. 

Consider $-1<p<0$. If $Z$ is nonempty then we may choose a point $z \in Z$ and let $\delta=\delta_z$ be the delta measure supported at that point. Using $(1-\e)\gamma+\e\delta$ as a trial measure for the energy of $K$, we know
\[
\int_K \! \int_K |x-y|^{-p} \, d\big((1-\e)\gamma+\e\delta\big) d\big((1-\e)\gamma+\e\delta\big) \leq V_p(K) = \int_K \! \int_K |x-y|^{-p} \, d\gamma d\gamma 
\]
for all $\e \in [0,1]$, with equality at $\e=0$. The coefficient at order $\e$ on the left side must therefore satisfy 
\[
\int_K \! \int_K |x-y|^{-p} \, d\gamma(x) d(\delta-\gamma)(y) \leq 0 ,
\]
or equivalently
\begin{equation} \label{eq:zineq}
\int_J |x-z|^{-p} \, d\gamma(x)  \leq V_p(J) . 
\end{equation}
The potential $y \mapsto \int_J |x-y|^{-p} \, d\gamma(x)$ is a continuous function of $y \in \R$ and tends to $\infty$ as $y \to \pm \infty$. Further, when $y$ lies to the left or right of $J$ the potential is concave (since $0<-p<1$ ensures concavity of the kernel). Hence the potential is strictly decreasing to the left of $J$ and strictly increasing to the right, while on $J$ it equals $V_p(J)$ by \cite[p.{\,}176]{BHS19}. Thus the potential at the point $z \notin J$ is strictly greater than $V_p(J)$, which contradicts \eqref{eq:zineq}. Hence no such point $z$ exists and thus $Z$ must be empty. 

Next consider $0<p<1$. The task is to show $Z$ has inner $p$-capacity zero. Suppose not, which means that some compact set $Z^\prime \subset K \setminus J$ has positive $p$-capacity. Let $\zeta$ be a probability measure on $Z^\prime$ that has finite $p$-energy. By arguing like above except with the inequality reversed and $\delta$ replaced by $\zeta$, we obtain 
\[
\int_K \! \int_K |x-y|^{-p} \, d\gamma(x) d(\zeta-\gamma)(y) \geq 0 ,
\]
or equivalently
\begin{equation} \label{eq:zineq2}
\int_{Z^\prime} \int_J |x-y|^{-p} \, d\gamma(x) d\zeta(y)  \geq V_p(J) . 
\end{equation}
The potential $y \mapsto \int_J |x-y|^{-p} \, d\gamma(x)$ is a continuous function of $y \in \R$ (see \cite[p.{\,}176]{BHS19}) and tends to $0$ as $y \to \pm \infty$, and is convex when $y$ lies to the left or right of $J$ (since $-p<0$ and so the kernel is convex). Hence the potential is strictly increasing to the left of $J$ and strictly decreasing to the right, while on $J$ it equals $V_p(J)$ by \cite[p.{\,}176]{BHS19}. Thus the potential at each point $y \in Z^\prime$ is strictly less than $V_p(J)$, contradicting \eqref{eq:zineq2}. Therefore no such set $Z^\prime$ exists and thus $Z$ must have inner capacity zero. 

Lastly, consider $p=0$. The argument of the previous paragraph applies again, using the logarithmic kernel instead of the Riesz kernel. 

\smallskip
``$\Longleftarrow$'' Suppose $0 \leq p<1$ and $K=J \cup Z$ for some nondegenerate interval $J$ and some set $Z$ with inner $p$-capacity zero. We may suppose $J$ and $Z$ are disjoint, so that $Z=K \setminus J$ and hence $Z$ is measurable and has inner capacity zero. Thus $\Length(Z)=0$ by \autoref{le:Zmzero} in the Appendix and so $\Length(J \cup Z) = \Length(J)$, while $\capp(J \cup Z) = \capp(J)$ by \autoref{le:Z}. Hence we may replace $K$ with $J$ in inequality \eqref{eq:caplength}, so that equality holds there.

\section{\bf Simplex examples, and proofs of \autoref{th:symmetrybreaking2dim} and \autoref{th:trianglecapformula}}
\label{sec:threepoint}

This section collects illustrative examples and proofs for $k$-point sets. 

\subsection*{Uniqueness for vertices of regular simplex (regular $(n+1)$-point set)} The capacity of a regular $k$-point set $S_k$ of diameter $d$ was found by Bj\"{o}rck \cite[p.{\,}264]{B56} to be 
\begin{equation} \label{eq:kpoint}
\capp(S_k) = \left( (k-1)/k \right)^{-1/p} d 
\end{equation}
when $p<0$, with unique equilibrium measure consisting of equal amounts of measure $1/k$ at each point of $S_k$.  For example, 
\[
\capp(S_2) = (1/2)^{-1/p} d , \qquad \capp(S_3) = (2/3)^{-1/p} d , \qquad \capp(S_{n+1}) = \left( n/(n+1) \right)^{-1/p} d .
\]
When $p \geq 0$ the capacity of $S_k$ is zero, since every measure on a finite set has infinite energy. 

Following is a short proof of Bj\"{o}rck's formula. Rescale to get diameter $1$, for simplicity. Write the points of the regular $k$-point set as $x_1,\ldots,x_k$, with distance $|x_i-x_j|=1$ between each pair of distinct points. Take $\mu$ to be a probability measure on the set and let $m_i$ be its measure at $x_i$, so that $\sum_i m_i=1$. The $p$-energy of $\mu$ is 
\begin{align*}
\sum_{i \neq j} m_i m_j  
& = \left( \sum_{i=1}^k m_i \right)^{\!\! 2} - \sum_{i=1}^k m_i^2 \\
& \leq \left( \sum_{i=1}^k m_i \right)^{\!\! 2} - \frac{1}{k} \left( \sum_{i=1}^k m_i \right)^{\!\! 2} \qquad \text{by Cauchy--Schwarz} \\
& = \frac{k-1}{k}   
 \end{align*}
 since $\sum_i m_i=1$. Equality holds if and only if each point carries the same amount of measure, meaning $m_i=1/k$ for each $i$, and thus that is the unique equilibrium measure. Raising the last expression to the power $-1/p$ gives Bj\"{o}rck's capacity formula for $S_k$. 

\subsection*{Nonuniqueness for vertices of nonregular simplex ($4$-point set)} When the simplex is nonregular, the equilibrium measure on the vertex set need not be unique, for $p<-2$. For example, construct a tetrahedron ($k=4$ vertices) as follows. Let $l=2^{-1/p}<2^{1/2}$, choose $h>0$ to satisfy $4h^2+(l^2/2)=1$, and define four points 
\begin{align*}
x_1 = (h,l/2,0) , & \quad 
x_2 = (h,-l/2,0) , \\
x_3 = (-h,0,l/2) , & \quad 
x_4 = (-h,0,-l/2) .
\end{align*}
The distances $d_{ij}=|x_i-x_j|$ are 
\[
d_{13} = d_{14} = d_{23} = d_{24} = 1 , \qquad d_{12} = d_{34} = l .
\]
Suppose $\mu$ is a measure on the four points with measure $m_i$ at $x_i$, so that $m_1+m_2+m_3+m_4=1$. The $p$-energy of this measure is 
\begin{align*}
\sum_{i \neq j} m_i m_j d_{ij}^{-p} 
& = 2(m_1 + m_2)(m_3 + m_4) + 2 l^{-p} (m_1 m_2 + m_3 m_4) \\
& \leq 2ab+(a^2+b^2) = (a+b)^2=1 ,
\end{align*}
where $a=m_1+m_2$ and $b=m_3+m_4$, so that $a+b=1$. Equality holds if $m_1=m_2=a/2$ and $m_3=m_4=b/2$, and so those choices give rise to equilibrium measures. Thus our $4$-point set $E=\{ x_1,x_2,x_3,x_4 \}$ has a whole one-parameter family of $p$-equilibrium measures, namely 
\[
\mu_a = \frac{a}{2} (\delta_1 + \delta_2) + \frac{1-a}{2} (\delta_3 + \delta_4) , \qquad a \in [0,1] ,
\]
where $\delta_i$ is the delta measure at $x_i$. The energy is $V_p(E)=1$ and $\capp(E)=1$.

\subsection*{Proof of \autoref{th:symmetrybreaking2dim}: regular three-point sets and disks}
\hypertarget{proofof:th:symmetrybreaking2dim}{For} the existence of the root $q_* \simeq -0.856$, observe that 
\[
2 \sqrt{\pi} \Gamma(1 - q/2) - 3 \Gamma((1 - q)/2) 
=
\begin{cases}
\pi-3>0 & \text{at $q=-1$,} \\
-\sqrt{\pi}<0 & \text{at $q=0$,} 
\end{cases}
\]
and so a root certainly exists between $-1$ and $0$. The root is unique since it satisfies 
\[
\frac{\log \Gamma(\frac{1-q}{2} + \frac{1}{2})  - \log \Gamma(\frac{1-q}{2})}{\frac{1}{2}} = 2 \log(3/2 \sqrt{\pi}) , 
\]
where the left side is a difference quotient for the strictly convex function $\log \Gamma$ and hence is strictly increasing with respect to $-q>0$. The numerical value $q_* \simeq -0.856$ is easily found by bisection or Newton's method. 

The capacity ratio of a disk $D \subset \R^2$ when $p \leq -2$ and $-2<q<0$ is  
\[
\frac{\capq(D)}{\capp(D)} = 2^{-1/p} \left( \frac{\sqrt{\pi} \Gamma(1-q/2)}{\Gamma((1-q)/2)} \right)^{\! \! 1/q} ;
\]
for example see \cite[Appendix A]{CL24b}. The capacity ratio of a regular three-point set $S_3$ is 
\[
\frac{\capq(S_3)}{\capp(S_3)} = (2/3)^{1/p-1/q}
\]
by the formula at the beginning of this section. 

Now suppose $p$ and $q$ satisfy the assumption in the theorem, that 
\[
-\infty < p < \frac{q \log(4/3)}{\log \frac{2 \sqrt{\pi} \Gamma(1 - q/2)}{3 \Gamma((1 - q)/2)}} , \qquad -2<q<q_* .
\]
Since $q<q_*$, our analysis above of the root $q_*$ shows that the denominator is positive. Thus the inequality may be rearranged to prove the desired conclusion of \autoref{th:symmetrybreaking2dim}, that the capacity ratio is larger for the regular three-point set:
\[
\frac{\capq(D)}{\capp(D)} < \frac{\capq(S_3)}{\capp(S_3)} .
\]

\subsection*{Proof of \autoref{th:trianglecapformula}: arbitrary three-point sets} 
\hypertarget{proofof:th:trianglecapformula}{Consider} a probability measure on the three points of $T$, with charges $x$ and $y$ lying distance $c$ apart, charges $y$ and $z$  separated by distance $a$, and charges $z$ and $x$ separated by distance $b$. The total measure is $x+y+z=1$. This measure has $p$-energy  
\[
F(x,y,z) = 2xyC+2yzA+2zxB = \begin{pmatrix} x & y & z \end{pmatrix} \begin{pmatrix} 0 & C & B \\ C & 0 & A \\ B & A & 0 \end{pmatrix} \begin{pmatrix} x \\ y \\ z \end{pmatrix}
\]
where for simplicity we have written $A=a^t,B=b^t,C=c^t$. Note $0 < A, B \leq C$. Letting $Q$ be the square matrix in this energy formula, we compute its inverse to be 
\[
Q^{-1} = \frac{1}{2ABC} \begin{pmatrix} -A^2 & AB & CA \\ AB & -B^2 & BC \\ CA & BC & -C^2 \end{pmatrix} .
\]

If $x, y$ or $z$ equals $0$ then $F(x,y,z) \leq C/2$, with equality if $z=0$ and $x=y=1/2$. (When $A,B<C$, equality holds only if $z=0$ and $x=y=1/2$.) We proceed to compare this ``boundary maximum'' value of the energy with a possible ``interior maximum''. 

Suppose now that $F$ has a critical point at $(x,y,z)$ in the region $x,y,z>0$, constrained by $x+y+z=1$. By taking the gradient of the energy and the constraint functional one finds, writing $\lambda$ for the Lagrange multiplier, that 
\[
\nabla F = 2 Q \! \begin{pmatrix} x \\ y \\ z \end{pmatrix} = \lambda \begin{pmatrix} 1 \\ 1 \\ 1 \end{pmatrix} .
\]
Solving for $x,y,z$ gives that 
 \begin{equation} \label{eq:xyz}
 \begin{pmatrix} x \\ y \\ z \end{pmatrix} = \frac{\lambda}{2} Q^{-1} \! \begin{pmatrix} 1 \\ 1 \\ 1 \end{pmatrix} = \frac{\lambda}{4ABC} 
 \begin{pmatrix} A(B+C-A) \\ B(C+A-B) \\ C(A+B-C) \end{pmatrix} .
 \end{equation}
Since $x>0$ and $B+C-A \geq B>0$, the first row in \eqref{eq:xyz} implies $\lambda>0$.  Since $z>0$, we deduce from the third row in \eqref{eq:xyz} that 
\[
A+B > C .
\]
Left-multiplying \eqref{eq:xyz} by $(1 \ 1 \ 1)$ and recalling the constraint $x+y+z=1$ enables us to solve for $\lambda$, finding  
\begin{equation} \label{eq:lagrangemult}
\frac{\lambda}{2} = \frac{2ABC}{2(AB+BC+CA)-(A^2+B^2+C^2)} .
\end{equation}
This denominator is positive, since it can be factored as 
\[
(\sqrt{A}+\sqrt{B}-\sqrt{C})(\sqrt{B}+\sqrt{C}-\sqrt{A})(\sqrt{C}+\sqrt{A}-\sqrt{B})(\sqrt{A}+\sqrt{B}+\sqrt{C}) 
\]
where the middle two factors are positive (recall $A,B \leq C$), and the first factor is positive too because $A+B>C$ implies that 
\[
\sqrt{A}+\sqrt{B}-\sqrt{C} \geq \frac{A+B}{\sqrt{C}} - \sqrt{C} > 0 .
\]

Summing up, we have shown that if $F$ has a constrained critical point where $x,y,z>0$, then the point is given by \eqref{eq:xyz} where necessarily $A+B>C$, and $\lambda$ is given by \eqref{eq:lagrangemult}. 

In the reverse direction, if $A+B>C$ then we may define $\lambda>0$ as in \eqref{eq:lagrangemult} and $x,y,z>0$ as in \eqref{eq:xyz}; the calculations above show that this $(x,y,z)$ is a constrained critical point. 

The energy at the constrained critical point equals
\[
\begin{pmatrix} x & y & z \end{pmatrix} Q \! \begin{pmatrix} x \\ y \\ z \end{pmatrix} = \begin{pmatrix} x & y & z \end{pmatrix} \frac{\lambda}{2} \begin{pmatrix} 1 \\ 1 \\ 1 \end{pmatrix} = \frac{\lambda}{2} .
\]
This energy strictly exceeds the boundary maximum energy $C/2$, since 
\begin{equation} \label{eq:lagrangemultagain}
\frac{\lambda}{2} = \frac{4AB \, (C/2)}{4AB-(A+B-C)^2} > \frac{C}{2}
\end{equation}
by rewriting the denominator in \eqref{eq:lagrangemult}. 

Hence if $A+B>C$ then the energy of $T$ is $\lambda/2$ as in \eqref{eq:lagrangemultagain}, and if $A+B \leq C$ then $A$ and $B$ are strictly less than $C$ and the energy of $T$ is $C/2$. In both cases, we have seen that the equilibrium measure on $T$ is unique. Finally, taking the $t$-th root of the energy gives the capacity formula stated in the proposition. 

\subsubsection*{Heron's formula} The denominator in \eqref{eq:lagrangemult} can be interpreted by Heron's formula as 
\[
16 \left( \text{area of triangle with sides $\sqrt{A}, \sqrt{B}, \sqrt{C}$} \right)^2.  
\]

\section{\bf Proof of \autoref{th:threeptratio}: optimizing capacity ratios amongst two- and three-point sets}
\label{sec:threeptratioproof}

If $T$ is a two-point set with distance $d$ between the points, then the $p$-capacity equals $(1/2)^{-1/p} d$ by a short computation (or take $k=2$ in \autoref{sec:threepoint}), and so the lower bound in the theorem is attained by every two-point set. If $T$ is a regular three-point set with distance $d$ between points, then the $p$-capacity equals $(2/3)^{-1/p} d$ by \autoref{sec:threepoint} and so the upper bound is attained. 

Suppose from now on that $T$ is a three-point set (not necessarily regular), with distances $a,b,c$ between the points and with $c$ the largest distance. Note $a+b \geq c$. For convenience, we rescale so that $c=1$, which means $0<a,b \leq 1$ and $a+b \geq 1$. The capacity ratio in the theorem depends only on $a$ and $b$: denote it by  
\[
R(a,b) = \frac{\capq(T)}{\capp(T)} .
\]
Write $r=-p$ and $s=-q$, so that 
\[
0<s<r .
\]
The goal is to prove $2^{1/r-1/s} \leq R(a,b) \leq (3/2)^{1/r-1/s}$. 

The proof breaks into five steps corresponding to the different parts of \autoref{fig:triangleparametrization}, which we recommend consulting, along with a concluding sixth step. The restriction $a+b\geq 1$ will be imposed only in that final step.

\subsubsection*{\textsc{Step 1}. $R(a,b)$ is constant in the region where $a+b \geq 1 $ and $a^{s}+b^{s}\leq 1$.} (This region is empty in case $0<s<1$.)

\textsc{Proof}. Since $0<a,b\leq 1$, we have $a^{r}+b^{r} \leq a^{s}+b^{s}\leq 1$ and so $R(a,b)=2^{1/r-1/s}$ is constant. Note that in this step and later ones, we use without comment the three-point set capacity formulas in \autoref{th:trianglecapformula}. 

\subsubsection*{\textsc{Step 2}. In the region where $a^{s}+b^{s}>1$ and $a^{r}+b^{r}< 1$, we have $\partial R/\partial a > 0$, so that $R$ is strictly increasing in $a$. By symmetry, also $\partial R/\partial b > 0$, so that $R$ is strictly increasing in $b$.} \

\textsc{Proof}. 
One computes $R(a,b) = 2^{1/r+1/s}/f(a,b)^{1/s}$ where  
\begin{align*}
f(a,b)
& = 4-(ab)^{-s}(a^{s}+b^{s}-1)^2 \\
& = 4 - b^{-s} \left( a^{s/2} - (1-b^s)a^{-s/2} \right)^{\! 2} ,
\end{align*}
where the term in parentheses is positive since $a^s+b^s>1$ by our assumptions in this step. Clearly $\partial f/\partial a<0$ since $1-b^s \geq 0$, and so $\partial R/\partial a > 0$ as claimed. Similarly $\partial R/\partial b > 0$. 

\subsubsection*{\textsc{Step 3}. Along the curve $a^{r}+b^{r}=1$, the ratio $R(a,b)$ is minimal at the endpoints and maximal at the middle point where $a=b$.} \ 

\textsc{Proof}. Note $a^{s}+b^{s}>1$. Recalling the formula for $R(a,b)$ from Step 2, we want to show along the curve $a^{r}+b^{r}=1$ that $f(a,b)$ is maximal at the endpoints and minimal at the point where $a=b$. Substituting $b=b(a)=(1-a^{r})^{1/r}$ yields that 
\[
f(a,b(a))=4-a^{-s}(1-a^{r})^{-s/r}(a^{s}+(1-a^{r})^{s/r}-1)^2 .
\]
Making the change of variables $z=a^{r}$, we have $f(a,b(a))=4-h(z)$ where
\[
h(z)=z^{-s/r}(1-z)^{-s/r}(z^{s/r}+(1-z)^{s/r}-1)^2 .
\]
By symmetry, it is enough to show $h$ is strictly increasing for $z \in (0,1/2)$. For simplicity, let $\gamma=s/r<1$. Then 
\[
h'(z)=\gamma z^{-\gamma-1}(1-z)^{-\gamma-1}(z^\gamma+(1-z)^\gamma-1)(z^\gamma-(1-z)^\gamma+1-2z) .
\]
The factor $z^\gamma+(1-z)^\gamma-1$ is positive since $a^{s}+b^{s}>1$. The final factor $k(z)=z^\gamma-(1-z)^\gamma+1-2z$ is also positive, because $k(0)=k(1/2)=0$ and $k''(z)=\gamma(\gamma-1)(z^{\gamma-2}-(1-z)^{\gamma-2})<0$ for $z\in (0,1/2)$, since $0<\gamma<1$. Hence $h^\prime(z)>0$ for $z\in (0,1/2)$, which completes this step.


\begin{figure}
\hspace{0.45cm}
\begin{tikzpicture}[scale=3.5]
\def\p{-4.4}
\def\q{-2.2}
\draw(0.5,-0.25) node[above] {$1\leq s<r$};
\filldraw[domain=0:1,smooth,variable=\x,samples=100,gray!10] plot (\x,{(1-\x^(-\q))^(-1/\q)});
\filldraw[gray!10] (1,0) -- (1,0.1) -- (0,1);
\draw[domain=0:1,smooth,variable=\x,samples=100,myorange, dashed] plot (\x,{(1-\x^(-\p))^(-1/\p)});
\draw[myorange] (1,0.7) node[right] {$a^{r}+b^{r}=1$};
\draw[domain=0:1,smooth,variable=\x,samples=100,mypurple, dashed] plot (\x,{(1-\x^(-\q))^(-1/\q)});
\draw[mypurple] (1,0.3) node[right] {$a^{s}+b^{s}=1$};
\draw[->, mymidgray] (0,-0.06) -- (0,1.1);
\draw (0,1.1) node[above] {$b$};
\draw[->, mymidgray] (-0.06,0) -- (1.1,0);
\draw (1.1,0) node[right] {$a$};
\draw (0,0.94) node[left] {$1$};
\draw (0.94,0) node[below] {$1$};
\draw[-,mymidgray] (1,-0.06) -- (1,1.06);
\draw[-] (1,0) -- (1,1);
\draw[->,arrows = {-Stealth[length=3pt, inset=2.3pt, width=6pt]}] (1,0) -- (1,0.8);
\draw[-, mymidgray] (-0.06,1) -- (1.06,1);
\draw[-] (0,1) -- (1,1);
\draw[->,arrows = {-Stealth[length=3pt, inset=2.3pt, width=6pt]}] (0,1) -- (0.8,1);
\draw[-] (0,1) -- (1,0);
\draw[->, arrows = {-Stealth[length=3pt, inset=2.3pt, width=6pt]}, dashed] ({2^(1/\p)},{2^(1/\p)}) --  (0.93,0.93);
\draw[-, dashed] (0.9,0.9) -- (1,1);
\draw[arrows = {-Stealth[length=3pt, inset=2.3pt, width=6pt]}] (0.88,0.6) -- (0.88,0.66);
\draw[arrows = {-Stealth[length=3pt, inset=2.3pt, width=6pt]}] (0.89,0.59) -- (0.95,0.59);
\draw[arrows = {-Stealth[length=3pt, inset=2.3pt, width=6pt]}] (0.77,0.78) -- (0.77,0.84);
\draw[arrows = {-Stealth[length=3pt, inset=2.3pt, width=6pt]}] (0.78,0.77) -- (0.84,0.77);
\draw[arrows = {-Stealth[length=3pt, inset=2.3pt, width=6pt]}] (0.6,0.88) -- (0.66,0.88);
\draw[arrows = {-Stealth[length=3pt, inset=2.3pt, width=6pt]}] (0.59,0.89) -- (0.59,0.95);
\draw[arrows = {-Stealth[length=3pt, inset=2.3pt, width=6pt]}] (0.8,0.94) -- (0.86,0.94);
\draw[arrows = {-Stealth[length=3pt, inset=2.3pt, width=6pt]}] (0.94,0.8) -- (0.94,0.86);
%
\draw[white, arrows = {-Stealth[color=black,length=3pt, inset=2.3pt, width=6pt]}] (0.75,0.927) -- (0.76,0.921);
\draw[white, arrows = {-Stealth[color=black,length=3pt, inset=2.3pt, width=6pt]}] (0.927,0.75) -- (0.921,0.76);
\end{tikzpicture}

\begin{tikzpicture}[scale=3.5]
\def\p{-2.3}
\def\q{-0.7}
\draw(0.5,-0.25) node[above] {$s<1< r$};
\draw[domain=0:1,smooth,variable=\x,samples=100,myorange, dashed] plot (\x,{(1-\x^(-\p))^(-1/\p)});
\draw[myorange] (1,0.55) node[right] {$a^{r}+b^{r}=1$};
\draw[domain=0:1,smooth,variable=\x,samples=100,mypurple, dashed] plot (\x,{(1-\x^(-\q))^(-1/\q)});
\draw[mypurple] (0,0.1) node[right] {$a^{s}+b^{s}=1$};
\draw[->, mymidgray] (0,-0.06) -- (0,1.1);
\draw (0,1.1) node[above] {$b$};
\draw[->, mymidgray] (-0.06,0) -- (1.1,0);
\draw (1.1,0) node[right] {$a$};
\draw (0,0.94) node[left] {$1$};
\draw (0.94,0) node[below] {$1$};
\draw[-,mymidgray] (1,-0.06) -- (1,1.06);
\draw[-] (1,0) -- (1,1);
\draw[->,arrows = {-Stealth[length=3pt, inset=2.3pt, width=6pt]}] (1,0) -- (1,0.7);
\draw[-, mymidgray] (-0.06,1) -- (1.06,1);
\draw[-] (0,1) -- (1,1);
\draw[->,arrows = {-Stealth[length=3pt, inset=2.3pt, width=6pt]}] (0,1) -- (0.7,1);
\draw[-] (0,1) -- (1,0);
\draw[->, dashed,arrows = {-Stealth[length=3pt, inset=2.3pt, width=6pt]}] ({2^(1/\p)},{2^(1/\p)}) --  (0.854,0.854);
\draw[-, dashed] (0.87,0.87) -- (1,1);
\draw[->,arrows = {-Stealth[length=3pt, inset=2.3pt, width=6pt]}] (0.56,0.57) -- (0.56,0.65);
\draw[->,arrows = {-Stealth[length=3pt, inset=2.3pt, width=6pt]}] (0.57,0.56) -- (0.65,0.56);
\draw[->,arrows = {-Stealth[length=3pt, inset=2.3pt, width=6pt]}] (0.35,0.78) -- (0.35,0.86);
\draw[->,arrows = {-Stealth[length=3pt, inset=2.3pt, width=6pt]}] (0.36,0.77) -- (0.44,0.77);
\draw[->,arrows = {-Stealth[length=3pt, inset=2.3pt, width=6pt]}] (0.78,0.35) -- (0.86,0.35);
\draw[->,arrows = {-Stealth[length=3pt, inset=2.3pt, width=6pt]}] (0.77,0.36) -- (0.77,0.44);
\draw[->,arrows = {-Stealth[length=3pt, inset=2.3pt, width=6pt]}] (0.68,0.9) -- (0.76,0.9);
\draw[->,arrows = {-Stealth[length=3pt, inset=2.3pt, width=6pt]}] (0.9,0.68) -- (0.9,0.76);
%
\draw[->,ultra thick, white, arrows = {-Stealth[color=black,length=3pt, inset=2.3pt, width=6pt]}] (0.55,0.879) -- (0.56,0.874);
\draw[->,ultra thick, white, arrows = {-Stealth[color=black,length=3pt, inset=2.3pt, width=6pt]}] (0.879-0.004,0.55+0.01) -- (0.874-0.004,0.56+0.01);
\end{tikzpicture}
\begin{tikzpicture}[scale=3.5]
\def\p{-0.8}
\def\q{-0.7}
\draw(0.5,-0.25) node[above] {$s<r\leq 1$};
\draw[domain=0:1,smooth,variable=\x,samples=100,myorange, dashed] plot (\x,{(1-\x^(-\p))^(-1/\p)});
\draw[myorange] (0.85,0.15) node[right] {$a^{r}+b^{r}=1$};
\draw[domain=0:1,smooth,variable=\x,samples=100,mypurple, dashed] plot (\x,{(1-\x^(-\q))^(-1/\q)});
\draw[mypurple] (0,0.1) node[right] {$a^{s}+b^{s}=1$};
\draw[->, mymidgray] (0,-0.06) -- (0,1.1);
\draw (0,1.1) node[above] {$b$};
\draw[->, mymidgray] (-0.06,0) -- (1.1,0);
\draw (1.1,0) node[right] {$a$};
\draw (0,0.94) node[left] {$1$};
\draw (0.94,0) node[below] {$1$};
\draw[-,mymidgray] (1,-0.06) -- (1,1.06);
\draw[-] (1,0) -- (1,1);
\draw[->,arrows = {-Stealth[length=3pt, inset=2.3pt, width=6pt]}] (1,0) -- (1,0.5);
\draw[-, mymidgray] (-0.06,1) -- (1.06,1);
\draw[-] (0,1) -- (1,1);
\draw[->,arrows = {-Stealth[length=3pt, inset=2.3pt, width=6pt]}] (0,1) -- (0.5,1);
\draw[-] (0,1) -- (1,0);
\draw[->, dashed,arrows = {-Stealth[length=3pt, inset=2.3pt, width=6pt]}] ({2^(1/\p)},{2^(1/\p)}) --  (0.74,0.74);
\draw[-, dashed] (0.5,0.5) -- (1,1);
%
\draw[->,arrows = {-Stealth[length=3pt, inset=2.3pt, width=6pt]}] (0.9,0.35) -- (0.9,0.45);
\draw[->,arrows = {-Stealth[length=3pt, inset=2.3pt, width=6pt]}] (0.76,0.35) -- (0.76,0.45);
\draw[->,arrows = {-Stealth[length=3pt, inset=2.3pt, width=6pt]}] (0.62,0.35) -- (0.62,0.45);
\draw[->,arrows = {-Stealth[length=3pt, inset=2.3pt, width=6pt]}] (0.35,0.9) -- (0.45,0.9);
\draw[->,arrows = {-Stealth[length=3pt, inset=2.3pt, width=6pt]}] (0.35,0.76) -- (0.45,0.76);
\draw[->,arrows = {-Stealth[length=3pt, inset=2.3pt, width=6pt]}] (0.35,0.62) -- (0.45,0.62);
%
\draw[->,ultra thick, white, arrows = {-Stealth[color=black,length=3pt, inset=2.3pt, width=6pt]}] (0.68-0.027,0.19+0.024) -- (0.67-0.027,0.2+0.024);
\draw[->,ultra thick, white, arrows = {-Stealth[color=black,length=3pt, inset=2.3pt, width=6pt]}] (0.19+0.01,0.68-0.01) -- (0.2+0.01,0.67-0.01);
\end{tikzpicture}
\caption{Parameter region for a three point set $T$ with side lengths $a,b,1$ where $0<a,b\leq 1$. The top diagram applies when $1\leq s < r$ (meaning  $p<q\leq -1$), the bottom left corresponds to $0<s<1< r$ (that is, $p<-1 <q<0$), and the bottom right to $0<s<r\leq1$ (equivalently $-1\leq p<q<0$). Arrows indicate directions of increase for the capacity ratio. The ratio is constant in the shaded region of the top diagram.} \label{fig:triangleparametrization}
\end{figure}
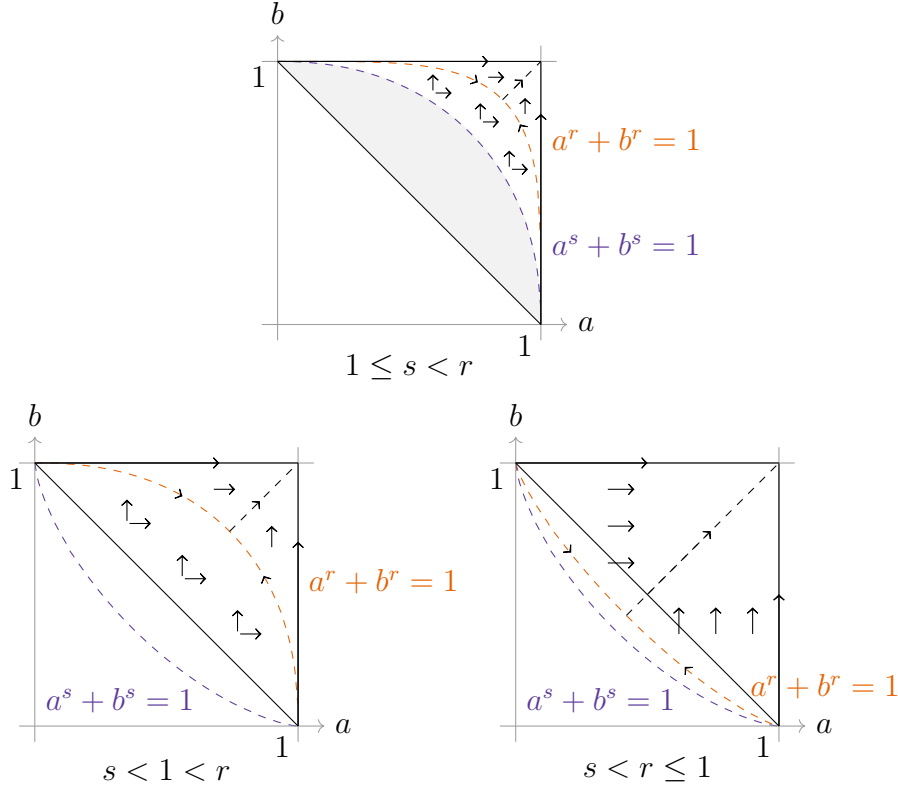


\subsubsection*{\textsc{Step 4}. $R(a,b)$ increases along the lines $a=1$ and $b=1$, and also along the line $a=b$ when $1/2<a^{r}<1$.} \ 

\textsc{Proof}. In all three situations, $a^r+b^r>1$ and $a^s+b^s>1$. 

When $a=1$ and $0<b<1$, we find $\capp(T)=2^{1/r}/(4-b^{r})^{1/r}$ and similarly for $\capq(T)$. Hence 
\[
R(1,b)=2^{1/s-1/r}\frac{(4-b^{r})^{1/r}}{(4-b^{s})^{1/s}} 
\]
and so 
\[
\frac{d\ }{db} \log R(1,b)= \frac{1}{b} \left( \frac{b^s}{4-b^s} - \frac{b^r}{4-b^r} \right) >0 ,
\]
since $b^s>b^r$. The calculation for $R(a,1)$ proceeds analogously.

When $a=b<1$ and $2a^{r}>1$, we have $\capp(T)=2^{1/r}a^2/(4a^{r}-1)^{1/r}$ and similarly for $\capq(T)$, and so
\[
R(a,a)=2^{1/s-1/r}\frac{(4a^{r}-1)^{1/r}}{(4a^{s}-1)^{1/s}} .
\]
Hence
\[
\frac{d\ }{da} \log R(a,a)= \frac{1}{a} \left( \frac{1}{4a^r-1} - \frac{1}{4a^s-1} \right) >0 ,
\]
since $1/2<a^r<a^s<1$.

\subsubsection*{\textsc{Step 5}.  In the region where $a^{r}+b^{r}>1$, if $0<a<b<1$ then $\partial R/\partial a > 0$ and so the ratio $R(a,b)$ is strictly increasing with respect to $a$, while if $0<b<a<1$ then $\partial R/\partial b > 0$ and $R(a,b)$ is strictly increasing with respect to $b$.}\ 

\textsc{Proof}. 
By symmetry, it suffices to consider the first claim, where $0<a<b<1$. 
The capacity ratio equals
\[
R(a,b) = 2^{1/s-1/r} \frac{(4(ab)^{r}-(a^{r}+b^{r}-1)^2 )^{1/r}}{(4(ab)^{s}-(a^{s}+b^{s}-1)^2 )^{1/s}} 
\]
and so
\[
\frac{\partial\ }{\partial a} \log R(a,b)=\frac{1}{a} \left( \frac{4(ab)^r-2a^r(a^r+b^r-1)}{4(ab)^{r}-(a^{r}+b^{r}-1)^2} -\frac{4(ab)^s-2a^s(a^s+b^s-1)}{4(ab)^{s}-(a^{s}+b^{s}-1)^2} \right).
\]
To show this quantity is positive, we fix $a$ and $b$ and aim to show $g(r)>g(s)$, where
\[
g(t)=\frac{4(ab)^t-2a^t(a^t+b^t-1)}{4(ab)^t-(a^{t}+b^{t}-1)^2}.
\]
Thus it suffices to prove that $g$ is strictly increasing for all $t>0$ such that $a^t+b^t>1$. 

Note the denominator of $g$ is positive by the proof of \autoref{th:trianglecapformula}, while the numerator is positive because it equals $2a^t(b^t-a^t+1)>0$.

Write $A=a^t, B=b^t$, so that  $A+B>1$ and $0<A<B<1$. We calculate that
\begin{align}
\frac{d}{d t} \log g(t) 
& = \frac{2}{t} (A \log A) \left( \frac{B-2A+1}{4AB-2A(A+B-1)}-\frac{B-A+1}{4AB-(A+B-1)^2} \right) \label{eq:aderivatives} \\ 
& + \frac{2}{t} (B \log B) \left( \frac{A}{4AB-2A(A+B-1)}-\frac{A-B+1}{4AB-(A+B-1)^2} \right). \label{eq:bderivatives} 
\end{align}
The expression in parentheses in \eqref{eq:aderivatives} is negative because cross multiplying gives
\begin{align*}
& (B-2A+1)(4AB-(A+B-1)^2)-(B-A+1)(4AB-2A(A+B-1) )\\ 
& = -(1-B)^2(1-A+B-A)-A^2(B+1)<0,
\end{align*}
since  $A<1$ and $A<B$.

Then, using $A \log A < B \log B$ (to be proved below), we may combine \eqref{eq:aderivatives} and \eqref{eq:bderivatives} to obtain that
\[
\frac{d }{d t} \log g(t) > \frac{2}{t} B \log B \left( \frac{B-A+1}{4AB-2A(A+B-1)} -\frac{2}{4AB-(A+B-1)^2} \right) .
\]
The last expression is positive as needed, since $\log B<0 $ and cross multiplying gives
\begin{align*}
& (B-A+1)(4AB-(A+B-1)^2)-2(4AB-2A(A+B-1)) \\
& = -4A(1-A)(1-B)-4AB(1-B) - (B-A+1)(A+B-1)^2 <0.
\end{align*}

To conclude, we check that $A \log A < B \log B$. Notice $B>(A+B)/2>1/2$. The function $x \mapsto x \log x$  is strictly increasing on the interval $[1/2,\infty)$ and so if $A\geq 1/2$ then $A \log A < B \log B$. Thus we may suppose from now on that $A<1/2$. Since $1/2<1-A<B$, it is enough to show $A \log A < (1-A)\log(1-A)$ when $A \in (0,1/2)$.

Let $c(x) = x \log x - (1-x) \log (1-x)$.
It is straightforward to check that $c''(x)>0$ on $(0,1/2)$, and so $c(x)$ is strictly convex. Since $c(0)=0=c(1/2)$, we conclude $c(x)<0$ on $(0,1/2)$. Hence $ c(A)<0$, as desired. 

\subsubsection*{\textsc{Step 6:} Conclusion} \
The relevant regions for the parameters $a$ and $b$ are graphed in \autoref{fig:triangleparametrization}, with arrows indicating how the ratio $R(a,b)$ increases. 

For the upper bound, clearly the ratio is strictly largest in the top right corner where $a=b=1$, meaning $T$ is a regular three-point set, with value $R(1,1)=(3/2)^{1/r-1/s}$. 

For the lower bound, we consider three cases corresponding to the diagrams in the figure. 

(i) (top) When $1\leq s < r$ the figure shows that $R(a,b)$ is smallest in the region where $a+b \geq 1 $ and $a^{s}+b^{s}\leq 1$, where it has the constant value $2^{1/r-1/s}$.

(ii) (bottom left) When $0<s<1< r$ the figure shows the ratio is smallest at some point on the line $a+b=1$, in which case the points of $T$ are collinear and $T$ can be regarded as a $1$-dimensional set. Applying \autoref{pr:onedim_lowermiddle} with $p$ and $q$ interchanged shows that the capacity ratio $R(a,b)$ is greater than the value $2^{1/r-1/s}$ for a two-point set.

(iii) (bottom right) When $0<s<r\leq1$ the figure shows that in the expanded region $a^r+b^r\geq 1$, $0\leq a,b \leq 1$, the function $R(a,b)$ is smallest in the limiting case of a two-point set ($a=1,b=0$, or $a=0,b=1$), with minimum value $2^{1/r-1/s}$.

\section{\bf Proof of \autoref{th:2deqtriangle}: planar sets when $p<0, q \leq-2$} \label{sec:2dtriangleproof}

First suppose $q<-2$. By the existence result \autoref{th:maximizer} with $n=2$, it is enough to prove \autoref{th:2deqtriangle} for two-point and three-point sets $T$. We will prove as desired that
\begin{equation} \label{eq:3ptsetmaxima}
\frac{\capq(T)}{\capp(T)}\leq 
\begin{cases}
(2/3)^{1/p-1/q}, & p<q<-2 , \\
2^{1/q-1/p}, & q<p<0, \quad q<-2 ,
\end{cases}
\end{equation}
with the maximum being attained by a regular three-point set in the first case and any two-point set in the second case. By interchanging $p$ and $q$ in the second case we see it is equivalent to $\capq(T)/\capp(T) \geq (1/2)^{1/p-1/q}$ when $p<q<0$ and $p<-2$. Thus the desired inequalities follow from \autoref{th:threeptratio}. 

Finally, if $q=-2$ the desired inequalities follow from letting $q \nearrow -2$ in \eqref{eq:3ptsetmaxima}, since $\capq(T)$ is a continuous function of $q<0$ \cite[Theorem 1.3]{CL24b}.

\subsection*{Nonuniqueness of equilibrium measures for \autoref{th:2deqtriangle}(b)}

With $p$ and $q$ suitably negative, \autoref{th:2deqtriangle}(b) shows that the capacity ratio is maximized by any two-point set, with equilibrium measure distributed equally at the two points. But that is not the only possibility. A maximizing set could support other equilibrium measures. For $q<p<-2$, we proceed to construct a set that attains the maximal capacity ratio while having two different $p$-equilibrium measures, which are supported at two points and three points respectively. The $q$-equilibrium measure in this example is supported at just two points. 

Let $r=-p$ and $s=-q$, so that $r,s>2$. Take $K = \{ x_0,x_1,x_2,x_3\} \subset \R^2$ to be the vertices of a ``kite'':
\[
x_0=(0,0),\,\, x_1=(1/2,\sqrt{3}/2), \,\, x_2=(-1/2,\sqrt{3}/2), \,\, x_3=(0,h),
\]
where the height $h=(4/3)^{1/r}$ is chosen so that two-point subset $\{x_0,x_3\}$ has the same $p$-capacity $(2/3)^{1/r}$ as the regular three-point subset $\{ x_0,x_1,x_2\}$, by \eqref{eq:kpoint} above. Note that $1<h<2/\sqrt{3}$ since $r>2$. We will show that these two- and three-point subsets provide the full $p$-capacity of $K$, so that that $\capp(K)=(2/3)^{1/r}$ and $K$ possess two distinct types of $p$-equilibrium measure. 

Any equilibrium measure is supported on at most three points, by \cite[Theorem 12]{B56}, since $n=2$ and $p<-2$. Thus to evaluate the capacity it suffices to consider all two- and three-point subsets of $K$. For the two-point sets, simply notice that the greatest distance between any pair of points is $h$, between $x_0$ and $x_3$. For the three point subsets, we have already considered the regular subset $\{ x_0,x_1,x_2\}$. The two remaining cases are $R=\{x_0,x_1,x_3\}$ (the right part of the kite) and $T = \{x_1,x_2,x_3\}$ (the top part of the kite). We will show using \autoref{th:trianglecapformula} that the $p$-capacities of $R$ and $T$ are less than or equal to $(2/3)^{1/r}$. 

The distances between points of $R$ are $1,l=(1/4+(h-\sqrt{3}/2)^2)^{1/2},h$. The middle distance $l$ is less than $1/\sqrt{3}$ since 
\[
l^2 < 1/4+(2/\sqrt{3}-\sqrt{3}/2)^2 = 1/3.
\]
Thus $h$ is the largest of the three distances, and since $r/2>1$ we have  
\[
1^r+l^r < 4/3 = h^r .
\]
Thus by \autoref{th:trianglecapformula}, the $p$-capacity is $2^{-1/r}(4/3)^{1/r}=(2/3)^{1/r}$ and so the subset $R$ has the same capacity as the two-point set $\{x_0,x_3\}$. As remarked after the theorem, the equilibrium measure of $R$ is supported at those two points and so does not provide any new type of equilibrium measure. 

For $T$, the distances between points are $l,l,1$. Since 
\[
l^r + l^r < \frac{2}{3} < 1^r ,
\]
\autoref{th:trianglecapformula} says that $\capp(T)=(1/2)^{1/r}$. This capacity is less than $(2/3)^{1/r}$ and so the subset $T$ does not yield the $p$-capacity of $K$. 

To show $K$ is a maximizer for the capacity ratio, we now compute the $q$-capacity of $K$. Again we need consider only two- and three-point subsets. The regular three-point set $\{x_0,x_1,x_2\}$ has $q$-capacity $(2/3)^{1/s}$, while the two-point set $\{x_0,x_3\}$ has $q$-capacity $2^{-1/s}(4/3)^{1/r}$. The two-point capacity is strictly larger, since $s>r$. And by the same calculations as used above for $p$, the sets $R$ and $T$ have $q$-capacity $2^{-1/s}(4/3)^{1/r}$ and $(1/2)^{1/s}$, respectively. Taking the largest of these numbers, we conclude the $q$-capacity is attained by the two-point subset, so that $\capq(K)=2^{-1/s}(4/3)^{1/r}$. 

Combining our findings, we see that $\capq(K)/\capp(K)=2^{1/r-1/s}$, which means $K$ is a maximizer for the capacity ratio in \autoref{th:2deqtriangle}(b), for $q<p<-2$.

\section{\bf Convergence of Riesz capacity wrt Hausdorff distance}
\label{sec:convergence}

For use in the next section, we now develop a convergence result for Riesz capacity when the sets converge in Hausdorff distance.  
\begin{proposition}[Hausdorff convergence and Riesz capacity] \label{pr:hausdorffdistance}
Let $n\geq 1$ and suppose $K_i$ is a sequence of compact sets in $\R^n$ that converges in Hausdorff distance to a compact, nonempty set $K$. If $p \geq 0$ then
\[
\limsup_{i \to \infty} \capp(K_i) \leq \capp(K).
\]
If $p<0$ then 
\[
\lim_{i \to \infty} \capp(K_i) = \capp(K).
\]
\end{proposition}
The ``$\limsup$'' result cannot be improved when $0 \leq p < n$. For example, if $K$ has positive $p$-capacity and $K_i$ is a finite subset that is $(1/i)$-dense in $K$ then $\dist(K_i,K) \leq 1/i \to 0$ but $\capp(K_i)=0$ for all $i$. The case $p \geq n$ is not interesting since the capacities are all zero (see \autoref{th:Rieszmonotonicity}(d)).  
\begin{proof}
Note that $K_i$ is nonempty for all large $i$, since the limiting set $K$ is nonempty. 

Consider the general situation of a symmetric, lower semicontinuous kernel $G(x,y)$ with energy 
\[
W(K) = \min_\mu \int_K \int_K G(x,y) \, d\mu d\mu 
\]
where $-\infty < W(K) \leq \infty$ and $\mu$ ranges over probability measures on the compact, nonempty set $K$; the minimum is known to be achieved for some equilibrium measure $\mu$, by \cite[Lemma 4.1.3]{BHS19}. We will prove $\liminf_{i \to \infty} W(K_i) \geq W(K)$, and that if $G$ is continuous then $\lim_{i \to \infty} W(K_i) = W(K)$. 

Assuming that conclusion has been established, we may apply it with kernel $G(x,y)$ equal to $|x-y|^{-p}$ when $p<0$, $\log 1/|x-y|$ when $p=0$, and $-|x-y|^{-p}$ when $p<0$, to obtain that $\liminf_{i \to \infty} V_p(K_i) \geq V_p(K)$ for $p \geq 0$ and $\lim_{i \to \infty} V_p(K_i) = V_p(K)$ for $p<0$, which proves the proposition. 

So suppose $L$ is a compact set that contains $K$ and all the $K_i$ and choose $\mu_i$ to be an equilibrium measure for $W(K_i)$. By the Helly selection principle, after passing to a subsequence, we may assume that $\mu_i$ converges weak-$*$ to some probability measure $\mu$ supported in $L$. This measure $\mu$ must in fact be supported in $K$, since if $U$ is an arbitrary open ball in the complement of $K$ and $\overline{U}\cap K$ is empty then the distance between $\overline{U}$ and $K$ is positive, and hence $\overline{U}\cap K_i$ is empty for all large $i$, which implies $\mu_i(U)=0$. Thus 
\begin{align*}
W(K) 
& \leq \int_K\int_K G(x,y) \, d\mu d\mu \\
& = \int_L\int_L G(x,y) \, d\mu d\mu \\
& \leq \liminf_{i\to \infty} \int_L\int_L G(x,y) \, d\mu_i d\mu_i \quad \text{by the Descent Principle \cite[Theorem 4.1.2]{BHS19}} \\
& =\liminf_{i\to \infty} W(K_i) ,
\end{align*}
as we wanted to show. 

Next suppose $G$ is continuous. Let $L \subset \Rn$ be compact and suppose $\e>0$. By continuity of the kernel and compactness of $L$ there exists $\delta>0$ such that 
\begin{equation} \label{eq:deltachoice}
|x-z| \leq 2\delta , \quad |y-z^\prime| \leq 2\delta , \quad x,y \in L \quad \Longrightarrow \quad G(x,y) \geq G(z,z^\prime) - \e .
\end{equation}
Consider now compact, nonempty subsets $A,B \subset L$ whose Hausdorff distance is less than $\delta$. We will show the energy $W(A)$ is bounded below by $W(B) - \e$. 
 
 Given a probability measure $\alpha$ on $A$, construct a probability measure $\beta$ on $B$ as follows: decompose $A$ into a finite, disjoint union of Borel sets $A_k$ each having diameter at most $\delta$; for each $k$, choose a point $a(k) \in A_k$ and some point $b(k) \in B$ that lies at most distance $\delta$ from $a(k)$; then define $\beta = \sum_k \alpha(A_k )\delta_{b(k)}$ to be a sum of weighted delta measures at those points, and note that $\beta$ is a probability measure on $B$ since $\sum_k \alpha(A_k)=\alpha(A)=1$. Since each point of $A_k$ lies at most distance $2\delta$ from $b(k)$, we see by the choice of $\delta$ in \eqref{eq:deltachoice} that 
\begin{align*}
\int_A \int_A G(x,y) \, d\alpha d\alpha 
& = \sum_{k,k^\prime} \int_{A_k} \int_{A_{k^\prime}} G(x,y) \, d\alpha d\alpha \\
& \geq \sum_{k,k^\prime} \int_{A_k} \int_{A_{k^\prime}} G\left( b(k),b(k^\prime) \right) d\alpha d\alpha - \e \\
& = \int_B \int_B G(x,y) \, d\beta d\beta - \e \geq W(B) - \e.
\end{align*}
Minimizing with respect to all measures $\alpha$ implies that $W(A) \geq W(B)-\e$. The same holds with $A$ and $B$ interchanged, and so the $W$-energy is continuous with respect to Hausdorff convergence of compact sets in $L$. In particular, returning to the original situation and taking $L$ large enough to contain $K$ and all the $K_i$, we conclude that $W(K_i) \to W(K)$ as $i \to \infty$. 
\end{proof}

\section{\bf Proof of \autoref{th:isodiametric}: capacitary Brunn--Minkowski and isodiametric inequalities} \label{sec:isodiametric}

The Brunn--Minkowski inequality for capacity is the key ingredient for proving \autoref{th:isodiametric}. It is known for $p=n-2$ and $p=n-1$: 
\begin{theorem}[Brunn--Minkowski for Riesz capacity; Borell \protect{\cite{B83,B84}}, Novaga--Ruffini \protect{\cite{NR15}}, Pommerenke \cite{Po59}] \label{th:capBM} Let $n \geq 2$. Suppose $A,B \subset \R^n$ are compact, convex and nonempty. If $p=n-2$ or $p=n-1$ then 
\[
\capp(\lambda A+(1-\lambda) B) \geq \lambda \capp(A) + (1-\lambda) \capp(B) , \qquad \lambda \in [0,1] .
\]
\end{theorem}
The planar logarithmic case of the theorem, where $n=2$ and $p=0$, is due to Pommerenke \cite[Satz 2]{Po59}. He employed a conformal mapping technique to handle connected compact sets, which are not assumed convex. Borell found a potential theoretic method for convex sets in two and higher dimensions \cite{B83}, \cite[Example 7.3]{B84}, that covers the logarithmic and Newtonian cases, $p=n-2$. For $p=n-1$ the theorem was proved by Novaga and Ruffini \cite{NR15}. 

Strictly speaking, the latter three authors worked with convex bodies, that is, convex sets not contained in a lower dimensional space. To allow degenerate sets $A$ or $B$ in the theorem above, one may thicken each set by adding an $\e$-ball and then letting $\e\to0$, as we now explain. Write $A_{\e}=A+\B^n(0,\e)$ for the convex body that arises by $\e$-thickening of a convex set $A$. We have   
\begin{align*}
& \capp\big((\lambda A +(1-\lambda) B)_{\e}\big) \\
& \geq \capp(\lambda A _{\e}+(1-\lambda) B_{\e}) && \text{since $(\lambda A +(1-\lambda) B)_{\e} \supset \lambda A _{\e}+(1-\lambda) B_{\e}$} \\
& \geq \lambda \capp(A_{\e})+(1-\lambda) \capp(B_{\e}) && \text{by \autoref{th:capBM} for convex bodies} \\
& \geq \lambda \capp(A)+(1-\lambda) \capp(B).
\end{align*}
Taking $\e \to 0$ completes the proof for $A$ and $B$ since on the left side of the inequality,  
\[
\capp(\lambda A +(1-\lambda) B) \geq \limsup_{\e \to 0} \capp((\lambda A +(1-\lambda) B)_{\e}) 
\]
by the Hausdorff convergence result in \autoref{pr:hausdorffdistance}. 

Thus, as observed by Muratov, Novaga and Ruffini \cite[p.{\,}1064]{MNR18}, the result of Novaga and Ruffini for $p=n-1$ can be obtained directly from Borell's result: regard the convex sets $A,B\subset \Rn$ as being convex sets in $\Rnp$ and apply \autoref{th:capBM} with $p=(n+1)-2$ (that is, Borell's result on $\Rnp$). 

\smallskip
It is natural to suppose the Brunn--Minkowski inequality should hold for a wider range of $p$. Novaga and Ruffini raised such a conjecture for $0<p<n$, but in view of the symmetry breaking expected in \autoref{fig:pqdiagram}, the interval $n-2<p<n$ seems a safer bet:
\begin{conjecture}[Capacitary Brunn--Minkowski; Novaga and Ruffini \protect{\cite[Conjecture 3.3]{NR15}}] \label{conj:capBM} 
\autoref{th:capBM} continues to hold for $p \in (n-2,n)$. 
\end{conjecture}
Taking the limit as $p \nearrow n$ would unify the family of capacitary Brunn--Minkowski inequalities with the classical Brunn--Minkowski inequality for volume, since volume appears as a limiting case of capacity (by \autoref{th:Rieszmonotonicity}(c)). 

\subsubsection*{Remark on equality cases}
When $p=n-2$, equality holds for convex bodies $A$ and $B$ if and only if they are homothetic, by work of Caffarelli, Jerison and Lieb \cite{CJL96}. When $p=n-1$, Muratov, Novaga and Ruffini \cite[Lemma 3.7]{MNR18} proved the equality case for $n=2$, by an argument that Qin and Zhang \cite{QZ24} observed extends easily to the case $n \geq 2$. We will not need these equality cases. 

\subsubsection*{Remark on variational capacities} Colesanti and Salani \cite{CS03} proved Brunn--Minkowski for the family of variational capacities, which are not the same as Riesz capacities except in the Newtonian case. 
An endpoint limit by Colesanti and Cuoghi \cite{CC05} extends this approach to a ``variational logarithmic capacity'', which is not the same as the logarithmic capacity in this paper. 

\subsection*{Proof of \autoref{th:isodiametric}}
First we handle $p=-\infty$, that is when the denominator of the ratio is diameter. Let $0 \leq q < n$, and take $\mathcal{F}$ to be the family of convex compact subsets of $\Rn$ with fixed diameter $d>0$. Choose a sequence $\{ K_i \}$ in $\mathcal{F}$ such that $\capq(K_i) \to M = \sup \{ \capq(K) : K\in \mathcal{F} \}$. By the Blaschke selection theorem, after passing to a subsequence we may suppose $K_i$ converges in Hausdorff distance to some convex set $K$. This set $K$ belongs to the family $\mathcal{F}$ because diameter is continuous with respect to Hausdorff distance. Moreover, $M = \lim_{i \to \infty}\capq(K_i)\leq \capq(K)$ by \autoref{pr:hausdorffdistance}. Therefore $K$ achieves the supremum $M$. 

Now one follows the standard argument for deducing an isodiametric inequality from Brunn--Minkowski: the convex set $K^\prime=(1/2)(K-K)$ is contained in the ball $\overline{B}$ of radius $d/2$ (diameter $d$) centered at the origin and so 
\[
\capq(\overline{B})\geq \capq(K')\geq \capq(K) ,
\]
where the final step relies on the Brunn--Minkowski \autoref{th:capBM}, and hence is valid when $q=n-2$ or $n-1$. 

Next, we handle $-\infty< p \leq-2$, by writing the ratio as a product of two ratios
\[
\frac{\capq(K)}{\capp(K)}=\frac{\capq(K)}{\diam(K)}\, \frac{\diam(K)}{\capp(K)}.
\]
We showed in the first part of the proof that the first ratio is maximal for the ball. The second ratio, $\diam(K)/\capp(K)$ is maximal for a two-point set by \autoref{le:twopoint}, but a two-point set and a ball having the same diameter necessarily possess the same $p$-capacity when $p\leq-2$ (see \cite[Theorem 4.6.6]{BHS19}). Since both ratios on the right hand side are maximal for the ball, so is $\capq(K)/\capp(K)$. 

\emph{Note.} \autoref{le:twopoint} is proved independently in the next section.

\section{\bf Additional proofs}
\label{sec:additionalproofs}

The proofs in this section are relatively short and direct. The only results needed from elsewhere in the paper are the elementary diameter bound in \autoref{th:Rieszmonotonicity}(b) (taken from \cite{CL24b}), the Hausdorff convergence result in \autoref{pr:hausdorffdistance}, and the explicit formula for the capacity of a regular $(n+1)$-point set in \autoref{sec:threepoint}. 

\subsection*{Proof of \autoref{pr:onedim_upperleft}}
\hypertarget{proofof_pr:onedim_upperleft}{By rescaling, }we may suppose without loss of generality that $K$ has diameter $2$. Then by a translation we may assume $K \subset [-1,1] = \overline{\B}^1$. Hence 
\[
\capq(K) \leq \capq(\overline{\B}^1) = \frac{\capq(\overline{\B}^1)}{2} \diam(K) .
\]
Note $\capp(K)=2^{1+1/p} =\capp(\overline{\B}^1)$ by \autoref{th:Rieszmonotonicity}(b), since $p \leq -1$ and $K$ and $\overline{\B}^1$ each have diameter $2$. Thus the last inequality is equivalent to saying 
\[
\frac{\capq(K)}{\capp(K)} \leq \frac{\capq(\overline{\B}^1)}{\capp(\overline{\B}^1)} .
\]
Equality holds if and only if $\capq(K) = \capq(\overline{\B}^1)$. Since $-1<q<1$, the equilibrium measures are unique and so (since the capacities are equal) the measure for $K$ equals the one for the interval $\overline{\B}^1$. That measure is supported on the whole interval (see \cite[Appendix A]{CL24b}), and so we conclude that equality occurs if and only if $K=\overline{\B}^1$. 

\subsection*{Proof of \autoref{pr:onedim_lowermiddle}}
\hypertarget{proofof:pr:onedim_lowermiddle}{Note} $\capq(K)=2^{1/q} \diam(K)$ and $\capp(K) \geq 2^{1/p} \diam(K)$ by \autoref{th:Rieszmonotonicity}(b), since $q \leq -1$ and $p<0$. Combining these facts proves the inequalities in the proposition. 

If $K$ is a two-point set then $\capp(K) = 2^{1/p} \diam(K)$, by taking $k=2$ in the beginning of \autoref{sec:threepoint}. Thus equality holds in the proposition when $K$ is a two-point set. 

The remaining task is to show when $-1<p<0$ that if $\capp(K) = 2^{1/p} \diam(K)$ then $K$ is a two-point set, or equivalently that if $K$ contains at least three points then $\capp(K) > 2^{1/p} \diam(K)$, or $V_p(K)>\diam(K)^{-p}/2$.

Suppose $K$ contains at least three points. By translating, we may assume these points are $-a<0<b$, where $a+b=\diam(K)$. Further, we may rescale $K$ so that $a+b=1$. Using the trial measure $\mu_s=(1-s)\delta_{-a}/2+s\delta_0+(1-s)\delta_b/2$, where $s\in[0,1]$, we find that 
\[
V_p(K)\geq (1-s)s a^{-p}+(1-s)s b^{-p}+(1-s)^2/2.
\]
Note that $s=0$ corresponds to the energy of the two-point set $\{-a,b\}$ of diameter $1$, which equals $1/2$, and so it is enough to show the derivative of the right side is positive at $s=0$. That inequality is straightforward:
\begin{align*}
\left. \frac{d\ }{ds} \right|_{s=0} (1-s)s a^{-p}+(1-s)s b^{-p}+(1-s)^2/2 
& = a^{-p}+b^{-p}-1 \\
& >a+b-1=0 ,
\end{align*}
since $0<-p<1$ and $0<a,b<1$. 

\subsection*{Proof of \autoref{th:maximizer}}
\hypertarget{proofof:th:maximizer}{Let} 
\[
M = \sup \, \{ \capq(K)/\capp(K) : \text{$K \subset \Rn$ is compact with $\capp(K)>0$} \} 
\]
and take a supremizing sequence $K_i$, meaning the capacity ratio for $K_i$ approaches $M$ as $i \to \infty$. By scale invariance of the ratio, we may normalize by requiring $\capp(K_i)=1$ for all $i$. By translation, we may suppose each set contains the origin. The bound in \autoref{th:Rieszmonotonicity}(b)  ensures the diameter of $K_i$ is bounded above by $2^{-1/p}$, for all $i$, and so the selection theorem for compact sets \cite[Theorem 2.7.9]{W94} implies that after passing to a subsequence, the $K_i$ converge in Hausdorff distance to some nonempty compact set $K$. This set has $\capp(K)=1$ by \autoref{pr:hausdorffdistance} (since $p<0$) and the same proposition yields (whether $q < 0$ or $q \geq 0$) that $\capq(K) \geq \lim_{i \to \infty} \capq(K_i) = M$. Thus $K$ achieves the maximum capacity ratio $M$. 

(The same argument applies in the case where $q=n$ since Lebesgue measure is upper semicontinuous with respect to Hausdorff distance on compact sets, that is $\Vol_n(K)\geq \limsup_{i\to\infty} \Vol_n(K_i)$.)

Now suppose $p<0$ and $q<-2$, so that Bj\"{o}rck's result \cite[Theorem 12]{B56} says the $q$-equilibrium measure of $K$ is supported on at most $n+1$ points, each point being extreme for the convex hull of $K$. That is, there is some $k \leq n+1$ and $E=\{x_1,\dots, x_k\}\subset K$ such that $\capq(E)=\capq(K)$. We know $K$ contains more than one point, since $\capp(K)=1$, and so $\capq(K)>0$ and hence $\capq(E)>0$; thus $E$ also contains more than one point, meaning $k \geq 2$. By monotonicity with respect to containment, $\capp(K) \geq \capp(E)$. Therefore,
\[
M = \frac{\capq(K)}{\capp(K)}\leq \frac{\capq(E)}{\capp(E)} \leq M .
\] 
Thus equality holds and the $k$-point set $E$ is a maximizer. 

Lastly, suppose $p<-1$ and write $\widehat{K}$ for the convex hull of the maximizing set $K$. This convex hull contains $K$ and so has larger or equal capacity: $\capq(\widehat{K}) \geq \capq(K)$ and $\capp(\widehat{K}) \geq \capp(K)$. Further, a result by Bj\"{o}rck \cite[Theorem 9]{B56} says since $p<-1$ that if $\mu$ is any equilibrium measure for $\widehat{K}$ then $\mu$ is supported in the set of extreme points of $\widehat{K}$. Those extreme points necessarily lie in $K$ and so $\mu$ is supported in $K$, which implies that $\capp(\widehat{K}) \leq \capp(K)$. Thus we have shown 
\[
\capq(\widehat{K}) \geq \capq(K) \quad \text{and} \quad \capp(\widehat{K}) = \capp(K) .
\]
Hence the capacity ratio for $\widehat{K}$ is greater than or equal to the ratio for $K$, which means the convex set $\widehat{K}$ is  a maximizing set too, and must in fact have the same $q$-capacity as $K$. 

\subsection*{Proof of \autoref{pr:isodiamsimplex}} 
\hypertarget{proofof:pr:isodiamsimplex}{Suppose} $q<-2$, so that by another result of Bj\"{o}rck \cite[Theorem 12]{B56}, any equilibrium measure $\mu$ of $K$ is supported at finitely many distinct points $x_1,\dots,x_k \in K$, where $2\leq k\leq n+1$. Let $E=\{x_1,\dots,x_k\}$, so that $\capq(K)=\capq(E)$. Since $\diam(E)\leq \diam(K)$ we have
\begin{equation} \label{eq:equality0}
\frac{\capq(K)}{\diam(K)}\leq \frac{\capq(E)}{\diam(E)}.
\end{equation}
The measure $\mu$ can be written
\[
\mu=m_1\delta_{x_1}+\dots+m_{k}\delta_{x_{k}}
\]
where the $m_i$ are positive and sum to $1$. Let $T=\{y_1,\dots,y_{n+1}\} \subset \Rn$ be a regular $(n+1)$-point set, that is, the vertex set of a regular simplex in $\Rn$, and choose $T$ to have the same diameter as $E$. Consider the probability measure $\nu$ on $T$ given by
\[
\nu = m_1\delta_{y_1}+\dots+m_k \delta_{y_k} .
\]
Then
\begin{align}
V_q(E)= \int_E \int_E |x-y|^{-q}\, d\mu d\mu 
& =  \sum_{i\neq j} m_i m_j|x_i-x_j|^{-q} \notag \\
&  \leq \sum_{i\neq j} m_i m_j(\diam{E})^{-q} \label{eq:equality1} \\
& = \int_T\int_T|x-y|^{-q}\, d\nu d\nu \leq V_q(T). \label{eq:equality2}
\end{align}
Hence $\capq(E)\leq \capq(T)$ and $\diam(E)=\diam(T)$, which completes the proof that the ratio of $q$-capacity over diameter is maximal for a regular $(n +1)$-point set, when $q<-2$. 

For the equality statement, suppose $K$ is a maximizer. Equality must hold in \eqref{eq:equality2} and so uniqueness of the equilibrium measure for the regular $(n+1)$-point set in \autoref{sec:threepoint} implies that $m_i=1/(n+1)$ for $i=1,\ldots,n+1$. Hence the set $E$ contains $k=n+1$ points. Equality must hold in \eqref{eq:equality1} and so $|x_i-x_j|=\diam(E)$ for all $i \neq j$, meaning $E$ is a regular $(n+1)$-point set. Note that $E$ and $K$ have the same diameter, since equality holds in \eqref{eq:equality0}. Thus the equilibrium measure $\mu$ with which we began the proof is the uniform probability measure on the vertices of a regular simplex having the same diameter and $q$-capacity as $K$. 

Now suppose $q=-2$. Let $T$ be a regular $(n+1)$-point set having the same diameter as $K$ and observe that 
\begin{align*}
\capnegtwo(K) 
& \leq \capq(K) && \text{when $q < -2$, by montonicity in \autoref{th:Rieszmonotonicity}(a),} \\
& \leq \capq(T) && \text{by the result proved above for $q<-2$} \\
& \to \capnegtwo(T) 
\end{align*}
as $q \nearrow -2$, using the formula from \autoref{sec:threepoint} that $\capq(T)=(n/(n+1))^{-1/q} \diam(T)$. For the equality statement when $q=-2$, we refer to Lim and McCann \cite[Corollary 2.2]{LM21}. 

\subsection*{Proof of \autoref{le:twopoint}}
\hypertarget{proofof:le:twopoint}{Suppose} $K \subset \Rn$ is compact and contains more than one point. Choose diametral points $a,b \in K$, meaning $|a-b|=\diam(K)$. The two-point set $\{a,b\}$ has the same diameter as $K$ while its $p$-capacity is at most that of $K$. Hence the ratio $\diam(K)/\capp(K)$ is less than or equal to the corresponding ratio for the two-point set. Each two-point set gives the same ratio, by the translation, rotation, and scaling properties of capacity. 

\subsection*{Proof of \autoref{le:maxinfty}}
\hypertarget{proofof:le:maxinfty}{Let} $0 \leq p < n$ and suppose $q<p$. If $q<0$ then we may simply take $K$ to be a set containing two points, so that $\capq(K)>0$ and $\capp(K)=0$, and of course $\Vol_n(K)=0$. 

If instead $q \geq 0$ then we let $d \in (q,p)$ and choose $K$ to be a compact set with Hausdorff dimension $d$: see for example the construction in Mattila \cite[Section 4.12]{M95}. This set has $\capq(K)>0, \diam(K)>0$, and $\capp(K)=0$, by \cite[Theorems 4.3.1 and 4.3.3]{BHS19}, while $\Vol_n(K)=0$ because $d<p<n$.

\section*{Acknowledgments}
Laugesen's research was supported by grants from the Simons Foundation (\#964018) and the National Science Foundation ({\#}2246537). Jeremy Tyson helpfully provided references regarding Hausdorff measure, and Jie Xiao pointed us to recent work on capacitary Brunn--Minkowski.

\appendix

\section{\bf Potential theoretic facts}
\label{sec:background}

A set $Z \subset \Rn$ is said to have inner $p$-capacity zero if $\capp(Z^\prime)=0$ for every compact $Z^\prime \subset Z$. 
\begin{lemma}[Inner $p$-capacity zero implies measure zero] \label{le:Zmzero}
Let $0 \leq p < n$. If $Z \subset \Rn$ is a measurable set with inner $p$-capacity zero then $\Vol_n(Z) = 0$. 
\end{lemma}
\begin{proof}
We establish the contrapositive. If $\Vol_n(Z)>0$ then $\Vol_n(Z^\prime)>0$ for some compact subset $Z^\prime \subset Z$ by inner regularity, and so $\capp(Z^\prime)>0$ by using normalized Lebesgue measure on $Z^\prime$ as a trial measure for the energy. Hence $Z$ does not have inner $p$-capacity zero. 
\end{proof}
\begin{lemma}[Inner $p$-capacity zero does not alter capacity] \label{le:Z}
Let $p \geq 0$ and $n \geq 1$. Suppose $K,Z \subset \Rn$ and that $K$ and $K \cup Z$ are compact. If $Z$ has inner $p$-capacity zero then 
\[
\capp(K) = \capp(K \cup Z) .
\] 
\end{lemma}
\begin{proof}
We may suppose $K$ and $Z$ are disjoint, by replacing $Z$ with the smaller set $Z \setminus K$. It suffices to show $\capp(K \cup Z) \leq \capp(K)$, since the other direction is immediate. We may suppose $\capp(K \cup Z)$ is positive.

We want to show $V_p(K \cup Z) \geq V_p(K)$. First consider $p>0$, and let $\mu$ be any probability measure on $K \cup Z$ with finite $p$-energy. We claim $\mu(Z)=0$. Suppose $\mu(Z)>0$, so that $\mu(\Rn \setminus K)>0$ and hence the open set $\Rn \setminus K$ contains some compact subset $\widetilde{Z}$ with positive $\mu$-measure. Then the compact set 
\[
Z^\prime=\widetilde{Z} \cap (K \cup Z) = \widetilde{Z} \cap Z \subset Z
\]
has $\mu(Z^\prime)=\mu(\widetilde{Z})>0$. Consequently $Z^\prime$ has finite $p$-energy and so positive $p$-capacity, contradicting that $Z$ has inner capacity zero. Hence $\mu(Z)=0$ as claimed and so $\mu$ is supported on $K$, so that $V_p(K \cup Z) \geq V_p(K)$ as desired. 

When $p=0$, one argues the same way except with logarithmic energy.
\end{proof}

\bibliographystyle{plain}

\end{document}